\documentclass[reqno,11pt]{amsart}
\usepackage{amsmath, latexsym, amsfonts, stmaryrd, amssymb, amsthm, amscd, array, caption}
\usepackage[toc,page]{appendix}
\usepackage[utf8]{inputenc}
\usepackage{graphics,epsf,psfrag,epsfig}
\usepackage[colorlinks=true, pdfstartview=FitV, linkcolor=blue, citecolor=blue, urlcolor=blue,pagebackref=false]{hyperref}
\usepackage{bm}

\usepackage[showlabels,sections,floats,textmath,displaymath]{}

\numberwithin{equation}{section}
\newtheorem{theorem}{Theorem}[section]
\newtheorem{lemma}[theorem]{Lemma}
\newtheorem{proposition}[theorem]{Proposition}

\newtheorem{rem}[theorem]{Remark}
\newtheorem{claim}[theorem]{Claim}

\renewenvironment{proof}[1][Proof]{\begin{trivlist}
\item[\hskip \labelsep {\bfseries #1}]}{\qed\end{trivlist}}

\newcommand{\ind}{\mathbf{1}}
\renewcommand{\ge}{\geq}
\renewcommand{\le}{\leq}

\renewcommand{\tilde}{\widetilde}
\renewcommand{\hat}{\widehat}

\DeclareMathSymbol{\leqslant}{\mathalpha}{AMSa}{"36} 
\DeclareMathSymbol{\geqslant}{\mathalpha}{AMSa}{"3E} 
\DeclareMathSymbol{\eset}{\mathalpha}{AMSb}{"3F}     
\renewcommand{\leq}{\;\leqslant\;}                   
\renewcommand{\geq}{\;\geqslant\;}                   
\newcommand{\dd}{\,\text{\rm d}}             

\newcommand{\bP}{{\ensuremath{\mathbf P}} }
\newcommand{\bE}{{\ensuremath{\mathbf E}} }

\newcommand{\bbN}{{\ensuremath{\mathbb N}} }

\newcommand{\bbR}{{\ensuremath{\mathbb R}} }

\newcommand{\bbZ}{{\ensuremath{\mathbb Z}} }

\newcommand{\ga}{\alpha}

\newcommand{\gd}{\delta}
\newcommand{\gep}{\varepsilon}       
\newcommand{\gp}{\varphi}

\makeatletter
\def\captionfont@{\footnotesize}
\def\captionheadfont@{\scshape}

\long\def\@makecaption#1#2{%
  \vspace{2mm}
  \setbox\@tempboxa\vbox{\color@setgroup
    \advance\hsize-6pc\noindent
    \captionfont@\captionheadfont@#1\@xp\@ifnotempty\@xp
        {\@cdr#2\@nil}{.\captionfont@\upshape\enspace#2}%
    \unskip\kern-6pc\par
    \global\setbox\@ne\lastbox\color@endgroup}%
  \ifhbox\@ne 
    \setbox\@ne\hbox{\unhbox\@ne\unskip\unskip\unpenalty\unkern}%
  \fi
  \ifdim\wd\@tempboxa=\z@ 
    \setbox\@ne\hbox to\columnwidth{\hss\kern-6pc\box\@ne\hss}%
  \else 
    \setbox\@ne\vbox{\unvbox\@tempboxa\parskip\z@skip
        \noindent\unhbox\@ne\advance\hsize-6pc\par}%
\fi
  \ifnum\@tempcnta<64 
    \addvspace\abovecaptionskip
    \moveright 3pc\box\@ne
  \else 
    \moveright 3pc\box\@ne
    \nobreak
    \vskip\belowcaptionskip
  \fi
\relax
}
\makeatother
\def\writefig#1 #2 #3 {\rlap{\kern #1 truecm
\raise #2 truecm \hbox{#3}}}

\newcommand{\simn}{\stackrel{n\to+\infty}{\sim}}

\newcommand{\simx}{ \stackrel{x\to +\infty}{\sim}}

\begin{document}

\title[Random Walks in the Cauchy domain of attraction]{Notes on Random Walks\\ in the Cauchy domain of attraction}

\author[Q. Berger]{Quentin Berger}
\address{LPSM, Sorbonne Universit\'e\\
Campus  Pierre et Marie Curie, case 188\\
4 place Jussieu, 75252 Paris Cedex 5, France}
\email{quentin.berger@sorbonne-universite.fr}

\date{}

\begin{abstract}
The goal of these notes is to fill some gaps in the literature about random walks in the Cauchy domain of attraction, which has often been left aside because of its additional technical difficulties. We prove here several results in that case: a Fuk-Nagaev inequality and a local version of it; a large deviation theorem; two types of local large deviation theorems. We also derive two important applications of these results: a sharp estimate of the tail of the first ladder epochs, and renewal theorems.
Most of our techniques carry through to the case of random walks in the domain of attraction of an $\alpha$-stable law with $\alpha \in(0,2)$, so we also present results in that case---some of them are improvement of the existing literature.
   \\[3pt]
  \\[3pt]
\end{abstract}

\maketitle

\section{Introduction}

This paper initiated when, consulting some colleagues about random walks in the Cauchy domain of attraction, they all shared the same observation that this case was often left aside in the literature, and that many very natural results were not proven (to the best of our knowledge). These notes therefore aim at filling as many gaps as possible.


\subsection{Setting and first notations}
\label{sec:notations}

Let $(X_i)_{i\ge 1}$ be a sequence of i.i.d.\ $\mathbb{Z}$-valued 
 random variables. We denote $S_n:= \sum_{i=1}^n X_i$, and $M_n:= \max_{1\leq i\leq n} \{ X_i\}$.
We assume that $(S_n)_{n\ge 0}$ is in the domain of attraction of an $\ga$-stable distribution, with $\alpha\in(0,2)$. We will put emphasis on the  case $\ga=1$, but we introduce notations in the general case.

We assume that there is some $\ga \in(0,2)$ and some slowly varying function $L(\cdot)$, such that as $x\to+\infty$, $\bP(|X_1|>x) \sim L(x) x^{-\ga}$ and
 \begin{equation}
\label{def:tail} 
 \bP(X_1>x) \sim p L(x)\, x^{-\ga} \,, \qquad \bP(X_1< -x) \sim q L(x)\, x^{-\ga} , 
 \end{equation}
with $ p+q=1$.  If $p=0$ (or $q=0$), then we interpret \eqref{def:tail} as $ o( L(x) x^{-\alpha})$. We stress that the condition~\eqref{def:tail} is equivalent to $(S_n)_{n\ge 0}$ being in the domain of attraction of an $\ga$-stable distribution (see \cite[IX.8, Eq.~(8.14)]{cf:Feller})

We define $a_n$ the scaling sequence, characterized up to asymptotic equivalence by the following relation ($\ga\in (0,2)$)
\begin{equation}
\label{def:an}
L(a_n) (a_n)^{-\ga} \sim \frac1n \quad \text{as } n\to\infty\, .
\end{equation}
We also define the centering sequence, that we denote $b_n$. We set $\mu:=\bE[X_1]$ if $X_1$ is integrable (if $X_1$ is not integrable we abbreviate it as $|\mu|=+\infty$), and we let
\begin{equation}
\label{def:bn}
\begin{split}
&b_n\equiv 0 \quad \text{if } \ga\in(0,1)\, ; \qquad \qquad  b_n =  n \mu \quad \text{if } \ga>1\, ;\\
&b_n = n \mu(a_n) \quad \text{with  } \mu(x) = \bE \big[X_1  \ind_{\{ |X_1|\leq x \}} \big] \quad \text{if } \ga=1 \, .
\end{split}
\end{equation}

We then have that $(S_n-b_n)/a_n$ converges in distribution to a non-trivial $\ga$-stable distribution, see e.g.\ \cite[IX.8]{cf:Feller} (in particular p.~315 (8.15) for the centering). We denote $g(\cdot)$ the density of the limiting distribution, which is continuous and bounded.
Under these assumptions,  Gnedenko's local limit theorem gives (see e.g.\ \cite[Ch.~9, \S50]{cf:Gnedenko}):
\begin{equation}
\label{eq:LLT}
\sup_{x\in\bbZ^d} \Big| a_n \bP\big(S_n  = x \big) - g\Big( \frac{x-b_n}{a_n} \Big) \Big| \stackrel{n\to+\infty}{\longrightarrow} 0 \, .
\end{equation}
This is sharp in the range when $|x-b_n|$ is of order $a_n$, but does not give much information when $|x-b_n|/a_n \to +\infty$: one aim of our paper is to provide large and local large deviations estimates, in particular in the case $\ga=1$ which was left aside in many cases.

\subsection{Organization of the paper and outline of the results}
Let us now present a brief overview of the paper.

In Section \ref{sec:largedev}, we present large and local large deviation results. First, our Theorem~\ref{thm:largedev} gives a standard large deviation estimate which seemed to be missing in the case $\ga=1$ in full generality. We provide a Fuk-Nagaev inequality and a local version of it (Theorem~\ref{thm:FukNagalpha1}), which is a cornerstone of our paper. We derive a local large deviation Theorem~\ref{thm:localLD} that extends Caravenna and Doney's result \cite[Thm.~1.1]{cf:CD} to the case $\ga=1$. Furthermore, with an additional locality assumption on the distribution of $X_1$, we provide an improved local large deviation Theorem~\ref{thm:localDoney}, extending that of Doney \cite[Thm.~A]{cf:Don97} to the case $\ga\in[1,2)$.

In Section \ref{sec:appli}, we give applications of these results. First, we consider the first descending ladder epoch $T_- = \inf \{ n ; S_n <0\}$, and give a sharp asymptotic of $\bP(T_- >n)$ in the case $\ga=1$ with infinite mean (Theorem~\ref{thm:ladder}). We also treat the finite mean case (Theorem~\ref{thm:ladder2}, done in \cite{cf:Don89} for $\ga\neq 1$). Some subtleties arise in the case $\ga=1$ with $\mu=0$, that we collect in Theorem~\ref{thm:laddermu0}.
Second, we consider renewal theorems for transient random walks: in the case $\ga=1$, we give sufficient conditions for the random walk to be transient, and we give the asymptotics of the Green function $G(x) = \sum_{n=0}^{+\infty} \bP(S_n =x)$ as $x\to+\infty$, see Theorem~\ref{thm:renewcentered} in the ``centered'' case $p=q$ and Theorem~\ref{thm:renewal} in the infinite mean case with $p\neq q$. We also present the result in the case of finite non-zero mean (Theorem~\ref{thm:renewalfinite}).

In Section \ref{sec:alpha1}, we develop further on the case $\ga=1$. In particular, we discuss the  question of the transience/recurrence of the random walk when $|\mu|=+\infty$: this is actually quite subtle, since the random walk, even if it may go to $+\infty$ or $-\infty$ in probability, is shown not to drift to $+\infty$ or $-\infty$ (at least when $p,q\neq 0$, see Proposition~\ref{prop:drift}). We give sufficient conditions for the random walk to be transient in the ``centered'' case $p=q$ (Proposition~\ref{prop:transient1}) or in the case $p\neq q$ (Proposition~\ref{prop:transient2}).
We also provide useful estimates on slowly varying functions (more precisely de Haan functions) related to the case $\ga=1$.

All the proofs are collected in Sections \ref{sec:FukNagaev}-\ref{sec:locallarge}-\ref{sec:ladder}-\ref{sec:renewals}.
In Section \ref{sec:FukNagaev} we state and prove Fuk-Nagaev inequalities, which are the central tool for proving the large and local large deviation estimates, that are derived in Section~\ref{sec:locallarge}.
In Section~\ref{sec:ladder}, we focus on  the ladder epochs theorems  in the case $\ga=1$ with $|\mu|=+\infty$ (and then adapt the proof to the case $\mu=0$).
In Section~\ref{sec:renewals}, we prove the renewal theorems, first in the case $\ga=1$ with infinite mean, and then in the finite mean case.

\begin{rem}
Let us mention that the assumption that $X_1 \in \mathbb{Z}$ is not essential, but is used for the sake of the presentation (mostly for statements of local results such as \eqref{eq:LLT}). For instance, if $X_1$ is \emph{non-lattice} (\textit{i.e.}\ if it is not supported by some lattice $a+b\mathbb{Z}$), then \eqref{eq:LLT} is replaced by Stone's local limit theorem~\cite{cf:Stone}: for any $h>0$
\begin{equation}
\label{eq:llt2}
\sup_{x\in\bbZ^d} \Big| \frac{a_n}{h} \bP\big(S_n  \in [x,x+h) \big) -  g\Big( \frac{x-b_n}{a_n} \Big) \Big| \stackrel{n\to+\infty}{\longrightarrow} 0 \, .
\end{equation}
All the rest of the paper can therefore easily be adapted to the case of a non-lattice distribution, using~\eqref{eq:llt2} in place of~\eqref{eq:LLT}---the only results needing some adaptations are  local results, for instance Theorem~\ref{thm:FukNagalpha1}-\eqref{eq:localLD}, Theorems~\ref{thm:localLD}-\ref{thm:localDoney} or Theorems~\ref{thm:renewcentered}-\ref{thm:renewal}-\ref{thm:renewalfinite}.
\end{rem}

\section{Large and local large deviations}
\label{sec:largedev}

Let us begin by stating a large deviation theorem which is standard when $\ga \neq 1$ (see \cite{cf:Nag79} and references therein, also \cite{cf:CH}), but appears to be missing in the case $\ga=1$.
\begin{theorem}
\label{thm:largedev}
Assume that \eqref{def:tail} holds with $\ga\in(0,2)$, and define $a_n$ as in \eqref{def:an} and $b_n$ as in \eqref{def:bn}. Then,
\[
\begin{split}
\bP(S_n - b_n > x) \sim n p L(x) x^{-\ga}, \qquad \text{ as } x/a_n \to +\infty \, ,\\
\bP(S_n - b_n <-  x) \sim n q L(x) x^{-\ga}, \qquad \text{ as } x/a_n \to +\infty \, .
\end{split}
\]
(If $p=0$ or $q=0$, one interpret this as $o(n L(x) x^{-\ga})$.)
\end{theorem}
This result  asserts that the large deviation is realized by a so-called \textit{one-jump} strategy (see  \cite{cf:DDS08} for a general setting). For the case $\ga=1$, \cite{cf:Nag79} or \cite{cf:DDS08} appear to be giving the correct behavior only when the step distribution is sufficiently centered, that is $\sup_{n} |b_n|/a_n <+\infty$ (see condition (27) in \cite{cf:DDS08}), or when  $x\ge \gd b_n$ for some $\gd>0$. The contribution of Theorem~\ref{thm:largedev} is therefore to extend the result in the case $\ga=1$ to the whole range $|x|/a_n \to+\infty$, without any restriction on $b_n$. 
In Section~\ref{sec:FukNagaev}, we  recall the central tool to prove this theorem, the so-called Fuk-Nagaev inequalities (we prove a new one in the case $\ga=1$): we then prove Theorem~\ref{thm:largedev} as a simple consequence of these inequalities.

\subsection{Local large deviations}

As far as a local version of Theorem \ref{thm:largedev} is concerned, results can be found in \cite[\S9]{cf:DDS08} in the ``centered'' case $\sup_{n} |b_n|/a_n <+\infty$.
Recently, Caravenna and Doney \cite[Thm.~1.1]{cf:CD} gave an improved local limit estimate in the case $\ga\in (0,1)\cup(1,2)$ (assuming $\mu=0$ when $\ga\in(1,2)$): given $\gamma>0$, they prove that there is some constant $C_0=C_0(\gamma)<+\infty$ such that for all $x\ge 0$,
\[\bP(S_n =x , M_n \le \gamma x) \le \frac{C_0}{a_n} \, \big( n \bP(X_1>x) \big)^{\lceil 1/\gamma \rceil}.\]
We extend their result to the case $\ga=1$:
we give a Fuk-Nagaev inequality and a local version of it, which are new.

\begin{theorem}
\label{thm:FukNagalpha1}
Assume that \eqref{def:tail} holds with $\ga=1$. For every $\gep>0$ there exist constants $c_1,c_2,c_3>0$ such that for any $1 \le y \le x$ and  $x\ge  a_n$, we have
 \begin{gather}
 \bP( S_n -  b_n  \ge x  ; M_n \le y) \le  e^{\frac xy}\Big( 1+ \frac{c_1 x}{ n L(y)} \Big)^{ - (1-\gep) \frac xy} +  e^{ - c_2\big( x/a_n \big)^{1/\gep} } \, 
 \label{eq:fuknag1}\\
 c_3 a_n \times \bP( S_n - \lfloor b_n \rfloor = x  ; M_n \le y)   \le e^{\frac{x}{y}} \Big( 1+ \frac{c_1 x}{ n L(y)} \Big)^{ -(1-\gep) \frac{x}{2y}} +  e^{ - c_2 \big( x/a_n \big)^{1/\gep} } \, . 
 \label{eq:localLD}
 \end{gather}
 This can be improved if we assume that $X_1$ has a symmetric distribution (see \eqref{eq:symmetric}), or that $X_1\geq 0$ (see Proposition~\ref{thm:alpha1}).
\end{theorem}
The corresponding bounds hold in the case $\ga\in(0,1)\cup (1,2)$, see Theorem~\ref{thm:fuknag} and Proposition~\ref{thm:localfuknag}. 
This has a simple consequence, which is an extension of the local limit theorem \eqref{eq:LLT} to the case when $|x|/a_n \to+\infty$.
\begin{theorem}
\label{thm:localLD}
For any $\ga \in(0,2)$ (the case $\ga\in (0,1)\cup (1,2)$ is in \cite[Thm. 1.1.]{cf:CD}), there exists a constant $C_0>0$ such that for any $x\in \bbZ$
\begin{equation}
\label{eq:localLDgen}
\bP(S_n - \lfloor b_n \rfloor =x) \le \frac{C_0}{a_n} \,  n L(|x|) (1+|x|)^{-\ga} \, .
\end{equation}
If $p=0$ ($q=0$ being symmetric),  we get
$\bP(S_n - \lfloor b_n \rfloor =x)  = o\big( \frac{1}{a_n} n L(x) x^{-\ga} \big)$ as $\frac{x}{a_n} \to +\infty$.
\end{theorem}

Let us underline that this statement is somehow optimal under the mere assumption \eqref{def:tail}: for any sequence $(\gep_n)_{n\ge 0}$ with $\gep_n \to 0$, one may find distributions verifying \eqref{def:tail} and a sequence $x_n$ such that $x_n/a_n \to+\infty$ with
\begin{equation}
\limsup_{n\to + \infty}  \frac{a_n \bP(S_n=x_n)}{\gep_n n L(x_n) x_n^{-\ga} } =+\infty \, .
\end{equation}

\subsection{Improved local large deviation}
 We may improve Theorem \ref{thm:localLD} if we assume that the  (left or right)  tail of the distribution of $X_1$ verifies a more local condition than \eqref{def:tail}, as considered for example by Doney in \cite{cf:Don97} in the case $\ga\in(0,1)$.
A natural condition is that there exists a constant $C_1>0$ (resp.~$C_2$) such that 
 \begin{gather}
 \label{hyp:localtail1}
\bP(X_1= x)  \le C_1 L(x) (1+x)^{-(1+\ga)}  \qquad \text{for all }  x\in \bbN\, , \\
\bP(X_1= -x)  \le C_2 L(x) (1+x)^{-(1+\ga)}  \qquad \text{for all }  x\in \bbN\, .
\label{hyp:localtail2}
 \end{gather}
Another stronger assumption, analogous to Eq.~(1.3) in \cite{cf:Don97}, is that $\bP(X_1 =x)$ (resp.\ $\bP(X_1=-x)$) is regularly varying, that is that
\begin{gather}
\label{hyp:doney}
\bP(X_1 = x) \sim  p \ga L(x) x^{-(1+\ga)} \qquad \text{ as } x\to+\infty \, , \\
\bP(X_1 = - x) \sim  q \ga L(x) x^{-(1+\ga)} \qquad \text{ as } x\to+\infty \,.
\label{hyp:doney2}
\end{gather}
(If $p=0$ resp.\ $q=0$, we interpret this as $o(L(x) x^{-(1+\ga)})$.)

\begin{theorem}
\label{thm:localDoney}
Assume that \eqref{def:tail} holds with $\ga\in(0,2)$.
If in addition we have \eqref{hyp:localtail1}, then there is a constant $C>0$ such that for any $x\ge a_n$,
\begin{equation}
\label{eq:localLdoney}
\bP(S_n - \lfloor b_n \rfloor = x) \le C n L(x) x^{-(1+\ga)} \sim  \frac{C}{x} \frac{L(x)}{L(a_n)} \left( \frac{x}{a_n} \right)^{-\ga}\, .
\end{equation}
If we have \eqref{hyp:doney}, then as $n\to +\infty$, $x/a_n\to+\infty$
\begin{equation}
\label{eq:localLdoney2}
\bP(S_n - \lfloor b_n \rfloor = x)  \sim n  p \ga L(x) x^{-(1+\ga)} \, .
\end{equation}
The analogous conclusion to \eqref{eq:localLdoney} (resp.~\eqref{eq:localLdoney2}) holds for $\bP(S_n - \lfloor b_n \rfloor = - x) $ if we assume \eqref{hyp:localtail2} (resp.~\eqref{hyp:doney2}).
\end{theorem}

\section{Applications: ladder epochs and renewal theorems}
\label{sec:appli}

In this section, we put more emphasis on the case $\ga=1$ (although we will also state results in the case $\ga\in(1,2)$).
Slowly varying functions will be interpreted  as  functions of the integers as well as differentiable functions of positive real numbers.
In the case $\ga=1$, we define
\begin{equation}
\label{def:ell}
\ell(n) := \int_1^n \frac{L(u)}{u} {\rm d} u   \ \ \text{if } |\mu|=+\infty \, ; \quad
\ell^\star(n) :=  \int_{n}^{+\infty} \frac{L(u)}{u} {\rm d} u \ \ \text{if } |\mu|<+\infty \, .
\end{equation}
We have that both $\ell(\cdot)$, $\ell^{\star}(\cdot)$ are slowly varying function (in fact, de Haan functions), with  $\ell(n)\to + \infty$, $\ell^\star (n) \to 0$ as $n\to+\infty$, and also $\ell(n)/L(n), \ell^\star(n)/L(n) \to +\infty$, see \cite[Prop.~1.5.9.a.]{cf:BGT}.

In the case $|\mu|=+\infty$, because of \eqref{def:tail} we have $\mu(t) \stackrel{t\to +\infty}{\sim} (p-q) \ell(t)$ and $b_n \simn (p-q) n  \ell(a_n)$. 
Similarly, in the case $|\mu|<+\infty$, we also get thanks to \eqref{def:tail} that  $\mu-\mu(t) \stackrel{t\to+\infty}{\sim} (p-q) \ell^\star(t)$: we end up with $b_n \sim \mu n$ if $\mu\neq 0$, and $b_n \sim (q-p) n \ell^{\star}(a_n)$ if $\mu=0$.
We therefore see that $|b_n|/a_n \to +\infty$ since $\ell(n)/L(n), \ell^\star(n)/L(n) \to +\infty$ (recall \eqref{def:an}), except possibly when $p=q$ and $\mu=0$ or $\pm\infty$.
In the case $p=q$ and $|\mu|=+\infty$ (resp.\ $\mu=0$), we get that $b_n = o(n\ell(a_n))$ (resp.\ $b_n = o( n \ell^\star(a_n))$,
and the general study gets much more subtle since we do not have a priori the asymptotic behavior of $b_n$. We will sometimes consider  the case $b_n/a_n \to b \in \bbR$.

\subsection{Ladder epochs}

Denote $T_-=\inf \{ n \, ;\, S_n <S_0 =0 \}$ and $T_+ = \inf \{ n \, ;\, S_n >S_0 =0 \}$ the first descending and ascending ladder epochs. A very natural question is first to know whether $T_-, T_+$ are defective (if so, then $S_n$ is said to \emph{drift} to $+\infty$, resp.\ $-\infty$), and to obtain the asymptotics of the tail probabilities $\bP(T_->n ), \bP(T_+ >n)$.

A crucial tool for this study is our Theorem \ref{thm:largedev}, which gives as an easy consequence the following precise asymptotics (see Lemmas \ref{lem:P<0} and \ref{lem:P<0bis} below):

(i) in the case $\ga=1$ with $|\mu|=+\infty$, we have $b_n/a_n \to +\infty$ if $p>q$ and $b_n/a_n \to -\infty$ if $p<q$, so that
\begin{equation}
\label{eq:forT>n}
\bP(S_n <0)   \sim \frac{q}{p-q} \frac{L( |b_n| )}{\ell( |b_n| )} \ \ \text{if } p>q \, ; \quad 
 \bP(S_n >0)  \sim \frac{p}{q-p} \frac{L(|b_n|)}{\ell(|b_n|)} \ \ \text{if } p<q \, ;
\end{equation}

(ii) in the case $\ga=1$ with $\mu=0$, we have $b_n/a_n \to -\infty$ if $p>q$ and $b_n/a_n \to +\infty$ if $p<q$, so that
 \begin{equation}
\label{eq:forT>n2}
\bP(S_n >0)   \sim \frac{p}{p-q} \frac{L( |b_n| )}{\ell^\star( |b_n| )} \ \ \text{if } p>q \, ; \quad 
 \bP(S_n <0)  \sim \frac{q}{q-p} \frac{L(|b_n|)}{\ell^\star(|b_n|)} \ \ \text{if } p<q \, .
\end{equation}
(One may naturally find asymptotics in the general case $\ga\in(0,2)$.)
From this, we may use Theorem 2 in \cite[XII.7]{cf:Feller} and Lemma \ref{lem:sumell} below, to get the following.
\begin{proposition}
\label{prop:drift}
Assume that \eqref{def:tail} holds with $\ga=1$, and assume that $|\mu|=+\infty$. If $q\neq 0$ then
\[\sum_{n=1}^{+\infty} \frac1n \bP(S_n<0)  = +\infty ,\]
and the random walk does not drift to $+\infty$, in the sense that $T_- <+\infty$ a.s.
Analogously, if $p\ne 0$ then the random walk does not drift  to $-\infty$, \textit{i.e.}\ $T_+ <+\infty$ a.s.
\end{proposition}

Note that in the case of finite non-zero mean $\mu\neq 0$, the strong law of large numbers gives that $(S_n)_{n\ge 1}$ is transient, and drifts to $+\infty$ (resp.\ to $-\infty$)  if $\mu>0$ (resp.\ $\mu<0$). In the case $\mu=0$, it is known that $(S_n)_{n\ge 1}$ is recurrent, so $T_-,T_+ <+\infty$ a.s.

Proposition~\ref{prop:drift} tells that if $\ga=1$ with $|\mu| = +\infty$ and $q\ne 0$, then even if $b_n/a_n \to +\infty$ so that $S_n$ goes to $+\infty$ in probability (more precisely $S_n/b_n$ converges in probability to~$1$), the random walk does not drift to $+\infty$ (and $\liminf S_n = -\infty$ a.s.).
This is due to the fact that even if the random walk is in probability ``close'' to $b_n \to+\infty$, once in a while a large jump to the left occurs (of length of order $b_n$), making $S_n<0$ (and in fact, of order $-b_n$). 
We will see below in Section \ref{sec:transient} that the random walk may still be transient (in the sense that $0$ may be visited only a finite number of times), but that determining  transience/recurrence is a more complicated matter.

\smallskip

As far as the asymptotics of $\bP(T_->n)$ are concerned (the estimates for  $T_+$ are symmetric by considering $(-S_n)_{n\geq 0}$), 
we refer to the seminal papers of Rogozin \cite{cf:Rog} and of Doney \cite{cf:Don82,cf:Don89}, and to \cite{cf:BorII} or \cite{cf:VW} for more recent results. 
However, the case $\ga=1$ (to our knowledge) does not appear to have been treated in the literature, apart from the recent work of Budd, Curien and Marzouk \cite[Prop.~1]{cf:BCM} where a rough estimate on $\bP(T_->n)$ is given in the case where $L(n)$ is constant. We shall give a sharp asymptotic, in the case of a general slowly varying function.

We give two levels of sharpness, depending on whether we assume that $L(\cdot)$ in \eqref{def:tail} is slowly varying in the Flajolet-Odlyzko sense (see conditions V1-V2 in \cite{cf:FO}), that is verifies:
\begin{itemize}
\item[V1.] there exists some $x_0>0$ and some $\phi \in (\pi/2,\pi)$ such that $L(z)$ is analytic in the region $ \{ z\, ;\,  \arg (z -x_0) \in [-\phi,\phi]  \}$
\item[V2.] we have, for any $\theta \in [-\phi,\phi]$ and $x\ge x_0$
\[ \Big| \frac{L(x e^{i\theta})}{L(x)} -1\Big| \le \gep(x) \, , \qquad  \Big| \frac{L(x \log x )}{L(x)} -1\Big| \le \gep(x) \, , \qquad \text{for some } \gep(x) \stackrel{x\to +\infty}{\to} 0\, .\]
\end{itemize}
This is satisfied for example if $L(x)$ is equal to $(\log x)^a$ or $(\log \log x )^a$for some $a\in \bbR$, but we stress that V2 fails for instance if $L(x) = \exp( (\log x)^b  )$ for some $b\in (0,1)$.

\subsubsection{Case $\ga=1$, $|\mu|=+\infty$}

\begin{theorem}
\label{thm:ladder}
Assume that \eqref{def:tail} holds with $\ga=1$ and $|\mu|=+\infty$.
Recall the definitions of $b_n= n \mu(a_n)$ in \eqref{def:bn} and of $\ell(\cdot)$ in \eqref{def:ell}.

{\rm (i)} If $p=q$ and  $b= \lim_{n\to +\infty} b_n /a_n$ exists, then there is a slowly varying $\gp(\cdot)$ such that 
\[\bP(T_->n)  = \varphi(n)\, n^{-\rho} \qquad \text{with } \rho = \frac12 + \frac1\pi \arctan\Big(\frac{2b}{\pi} \Big) \, .\]

{\rm (ii)} If $p < q$, then $b_n \sim -(q-p) n \ell(a_n) \to -\infty$ as $n\to +\infty$, and
\[\bP(T_- >n)  =
  \frac{L(|b_n|)}{n} \ell(|b_n|)^{\frac{p}{q-p} -1 +o(1)}  \, .
\]

{\rm (iii)} If $p>q$, then $b_n \sim (p-q) n \ell(a_n) \to +\infty$ as $n\to+\infty$, and
\[  \bP(T_->n) = \ell(b_n)^{ - \frac{q}{p-q} +o(1)} \, . \]
If additionally V1-V2 above holds, then we can make the $o(1)$ in {\rm (ii)} and {\rm (iii)} more precise: in both cases there is a slowly varying $\tilde L(\cdot)$ such that $\ell(|b_n|)^{o(1)}$ can be replaced by~$\tilde L\big( \ell(|b_n|)\big)$.
\end{theorem}
We are also able to deal with the case $p=q=1/2$ when $b_n/a_n \to+\infty$ but we did not state it here for conciseness, since it requires further notations: we refer to Section~\ref{sec:p=q} for details, see in particular \eqref{eq:T-T+}.
After this article was complete, Kortchemski and Richier~\cite{cf:KR} proved a similar statement by replacing the $\ell(|b_n|)^{o(1)}$ by some slowly varying function $\Lambda(n)$, without assuming V1-V2, see \cite[Prop.~12]{cf:KR}. However, 
our statement under assumption V1-V2 is somehow stronger: it tells that $\Lambda(n)$ is of the form $\tilde L ( \ell(|b_n|))$ with $\tilde L(\cdot)$ slowly varying.

As an application of Theorem \ref{thm:ladder}, we improve Proposition~1 of \cite{cf:BCM} in the case $p\ne q$: if $L(n)$ is constant equal to~$c$ (and obviously verifies V1-V2), then $\ell(n) \sim c \log n$, $a_n \sim c n$, and $b_n \sim (p-q) c n\log n$. Hence, we obtain that there exist some slowly varying functions $\bar L(\cdot),\tilde L(\cdot)$ such that 
\begin{equation}
\label{eq:T>n}
\bP(T_->n) =
\begin{cases}
(\log n)^{ - \frac{q}{p-q}} \, \bar L(\log n)& \text{if } p>q \, ,\\
n^{-1} (\log n)^{ \frac{p}{q-p} -1 }\,  \tilde L(\log n) & \text{if } p<q \, .
\end{cases}
\end{equation}

\subsubsection{Case $|\mu|<+\infty$}
In the finite mean case, estimates for $\bP(T_- >n)$ can be found in \cite{cf:Don89} and in \cite{cf:BorI,cf:BorII}, but again, the case $\ga=1$ appears to have been left aside. 
We stress that the proofs in \cite{cf:BorI, cf:BorII} rely on an asymptotic estimate of $\bP(S_n<0)$ (or $\bP(S_n>0)$) as $n\to+\infty$, that are given by Theorem~\ref{thm:largedev}.
We state the results for the sake of completeness, leaving aside the case $\ga=1$, $\mu=0$ for the moment.

\begin{theorem}[cf.\ Theorems 0-I of \cite{cf:Don89} and Theorem 1 of \cite{cf:BorI}]
\label{thm:ladder2}
Assume that \eqref{def:tail} holds with $\ga\in[1,2)$ and $|\mu|<+\infty$.

{\rm (i)} If $\mu > 0$, then
$\bP(T_-  = +\infty) =e^{-D_-}<1$ with $D_- := \sum_{k=1}^{+\infty} k^{-1}\bP(S_k <  0) <+\infty$. Moreover, if $q\neq 0$ we have
\[
\bP(T_- = n ) \simn  \frac{q\, e^{-D_-}}{\mu^{\ga}} L(n) n^{-\ga} \, .
\]
 
{\rm (ii)} If $\mu < 0$ and $p\neq 0$, then we have $D_+:= \sum_{k=1}^{+\infty} \frac{1}{k} \bP(S_k \ge 0) <+\infty$ and
 \[\bP(T_- > n) \simn   \frac{ p\, e^{D_+}}{ |\mu|^{\ga}}\, L(n) n^{-\ga}  \, .\]
 
{\rm (iii)} If $\mu =0$ and $\ga\in(1,2)$, then there exists a slowly varying function $\gp(\cdot)$ such that
 \[\bP(T_- >n) = \gp(n) \, n^{-\rho} \, , \]
 with $\rho := \frac12 +\frac{1}{\pi \ga} \arctan\big( (p-q) \tan(\pi\ga/2) \big)$.
\end{theorem}
The proof of items (i)-(ii) easily translates from \cite[Thm.~1]{cf:BorI} to the case $\ga=1$ with a finite non-zero mean.
For item (i), we simply use that in the case $\mu>0$, we have
$n^{-1}\bP(S_n<0)\sim q L(n) (\mu n)^{-\ga}$ (from Theorem~\ref{thm:largedev}), which is regularly varying:  the result then follows from Eq.~(23) in \cite{cf:BorI} together with an application of \cite[Thm.~1]{cf:CNW}, using that $D_-:=\sum_{k} k^{-1} \bP(S_k>0) <+\infty$.
Item (ii) is similar. For (iii), we have that $b_n\equiv 0$, and we can use that $\bP(S_k >0) = \bP(S_k/a_k>0)$ converges to $\bP(Y>0)$ where $Y$ has the limiting $\ga$-stable law, with skewness parameter $\beta =p-q$ (hence the formula for the positivity parameter $\rho:=\bP(Y>0)$, see \cite[Sec.~2.6]{cf:Zolot}). This implies (see \cite{cf:Rog,cf:Don89}) that $T_-$ is in the domain of attraction of a positive stable random variable with index $\rho$, and item (iii) follows.

\subsubsection{Case $\ga=1$, $\mu=0$}
We left that case aside since it is not a straightforward adaptation of \cite[Thm.~1]{cf:BorI}, in particular because $\sum_{k} k^{-1} \bP(S_k>0) =+\infty$ (see Proposition~\ref{prop:drift}).
\begin{theorem}
\label{thm:laddermu0}
Assume that \eqref{def:tail} holds with $\ga=1$ and assume that $\mu=0$. Then  (recall the definition \eqref{def:ell} of $\ell^\star(\cdot)$),
 \begin{enumerate}
 \item[(i)] If $p>q$ we have $b_n\sim -(p-q) n\ell^{\star} (a_n) \to -\infty$ as $n\to+\infty$, and
 \[\bP(T_- >n) = \frac{L(|b_n|)}{n} \ell^{\star}(|b_n|)^{ - \frac{p}{p-q} -1 +o(1)} \, . \]
 \item[(ii)] If $p<q$ we have $b_n\sim (q-p) n\ell^{\star} (a_n) \to +\infty$ as $n\to+\infty$, and
  \[\bP(T_- >n) = \ell^{\star}(b_n)^{ \frac{q}{q-p}+o(1)} \, . \]
 \end{enumerate}
If additionally $L(\cdot)$ verifies  V1-V2 above, then in each case there is a slowly varying  $\tilde L(\cdot)$ such that such that $\ell^\star(|b_n|)^{o(1)}$ can be replaced by $\tilde L\big( \ell^\star(|b_n|)\big)$.
\end{theorem}
We prove this theorem in Section~\ref{sec:mu=0}, where we discuss also the case $p=q$---it is treated similarly to Theorem~\ref{thm:ladder}-(i), see in particular \eqref{eq:pqmu0}.

\subsection{Renewal theorems}

An interesting application of the local limit Theorems~\ref{thm:localLD}-\ref{thm:localDoney} is that we are able
to obtain renewal theorems for transient random walks $(S_n)_{n\ge 1}$: we give the behavior, as $x\to+\infty$, of the Green function $G(x) = \sum_{n} \bP(S_n=x)$.

When the step variable $X_1$ is positive, $(S_n)_{n\ge 0}$ is a renewal process, and we write $S=\{S_0, S_1, S_2, \ldots\}$ the set of renewal points (with a slight abuse of notations).
Then $G(x)$ is interpreted as the renewal mass function $ \bP(x \in S)$, and has been studied in a variety of papers. The well-known renewal theorem gives that whenever $X_1 > 0$ and $\mu=\bE[X_1]<+\infty$, then $\bP(x\in S) \to 1/\mu$ as $x\to+\infty$. Assuming additionally that \eqref{def:tail} holds with $\ga\in (0,1]$ (and necessarily $p=1,q=0$ since $X_1 > 0$), then 
Garcia and Lamperti \cite{cf:GL} showed the strong renewal theorem
\begin{equation}
\label{eq:GL}
\bP(x\in S) \stackrel{x\to+\infty}{\sim}   \frac{\alpha\sin(\pi \alpha)}{\pi}  L(x)^{-1} x^{-(1-\ga)} \qquad \text{ if }\ga\in(1/2,1) ;
\end{equation}
Erickson \cite{cf:Eric} also proved that, 
\begin{equation}
\label{eq:Eric}
\bP(x\in S) \stackrel{x\to+\infty}{\sim} \frac{1}{\mu(x)} \qquad  \text{ if }  \ga=1  \text{ with } |\mu|=+\infty . 
\end{equation}
Finally,  Caravenna and Doney \cite{cf:CD} gave very recently a necessary an sufficient condition for the above strong renewal theorem \eqref{eq:GL} to hold when $\ga\in(0,1/2]$. 

\smallskip

When $(S_n)_{n\ge 1}$ is a (general) random walk rather than a renewal process, the Green function $G(x)$ has been considered in the case $\alpha\in(0,1)$ for example in \cite{cf:CD,cf:Will}, and when $X_1$ has a finite mean (see Theorem~\ref{thm:renewalfinite} below).
We shall prove renewal theorems in the case $\ga=1$, under some specific assumptions that ensures the transience of the random walk. We comment further in Section \ref{sec:alpha1} the particularity of the case $\ga=1$, $|\mu|=+\infty$, where even the question of recurrence/transience of $(S_n)_{n\ge1}$ is subtle.

The first renewal theorem we get is in the ``centered'' case $p=q=1/2$.
\begin{theorem}
\label{thm:renewcentered}
Assume that \eqref{def:tail} holds with $\ga=1$ and $p=q=1/2$, and that $b:=\lim_{n\to+\infty} b_n /a_n$ exists, $b\in\bbR$. Assume also that $\sum_{n\ge 1} \frac{1}{n L(n)} <+\infty$. Then $S_n$ is transient, and
\[G(x) \stackrel{n\to+\infty}{\sim}  \frac{2}{\pi (1+(2b)^2)} \sum_{n>x} \frac{1}{n L(n)} \, , \]
which is  a slowly varying function vanishing as $x\to+\infty$.
\end{theorem}

In the case where \eqref{def:tail} holds with $\ga=1$, $|\mu|=+\infty$ and $p\neq q$, we need the extra assumptions \eqref{hyp:localtail1}-\eqref{hyp:localtail2} to be able to derive a renewal theorem ---otherwise it is not even clear if $(S_n)_{n\ge 0}$ is transient, see Proposition~\ref{prop:transient2}. Recall that in that case, we have $\mu(x) \sim (p-q) \ell(x)$ (and goes to $+\infty$ if $p>q$ and $-\infty$ if $p<q$).
\begin{theorem}
\label{thm:renewal}
Assume that \eqref{def:tail} holds with $\ga=1$ $|\mu| = +\infty$, and $p\neq q$. Assume additionally that  \eqref{hyp:localtail1}-\eqref{hyp:localtail2} hold. Then $S_n$ is transient, and we have that $G(x) = O \big( 1/|\mu(x)|\big)$ and also $\liminf G(x)/\mu(x) \ge 1 $ if $p>q$.

\begin{enumerate}
\item[(i)] If $p>q$ and $\bP( X_1 = -x) \stackrel{x\to+\infty}{\sim}  q L(x) x^{-2}$ (if $q=0$, we interpret this as $o( L(x) x^{-2})$), then 
\begin{equation}
\label{eq:renewalthm}
G(x)\stackrel{x\to+\infty}{\sim} \frac{1+q/(p-q)}{\mu(x)} \, . 
\end{equation}

\item[(ii)] If $p<q$ and $\bP( X_1 = x) \stackrel{x\to+\infty}{\sim}  p L(x) x^{-2}$ (if $p=0$, we interpret this as $o( L(x) x^{-2})$), then 
\begin{equation}
G(x)\stackrel{x\to+\infty}{\sim} \frac{p/(q-p)}{|\mu(x)|} \, . 
\end{equation}
\end{enumerate}
\end{theorem}
We therefore recover Erickson's result \eqref{eq:Eric} in the case of general random walks, at the expense of assumptions \eqref{hyp:localtail1}-\eqref{hyp:localtail2} (in the case of renewals we get $q=0$).

Let us make a short comment on \eqref{eq:renewalthm}. When $p>q$ (so that $b_n \to+\infty$) the behavior in \eqref{eq:renewalthm} comes from two types of  contribution to  $G(x) = \sum_{k=1}^{+\infty} \bP(S_k=x)$: when the number of steps is of the order of $k_x$ verifying $b_{k_x} = x$ (so that $S_{k_x} $ is approximately $x$), and when the number of steps is much larger ($S_k$ is much larger than $x$, but large jumps to the left still occur). For the first part, we prove that it is asymptotic to $1/\mu(x)$. For the second part, we are able to prove that it is $O(1/\mu(x))$ under \eqref{hyp:localtail2}, and we can get its asymptotic behavior if $\bP( X_1 = -x) \stackrel{x\to+\infty}{\sim}  q L(x) x^{-2}$.

\smallskip
Finally, for the sake of completeness, we state a renewal theorem in the case of a finite non-zero mean.
Recall that if $X_1$ is integrable with $\mu=0$, then the walk is recurrent, \textit{i.e.}\ $G(x) = +\infty$ for all $x\in \bbZ$.

\begin{theorem}
\label{thm:renewalfinite}
Assume that $X_1$ is integrable, with $\mu=\bE[X_1]\neq 0$.

{\rm (i)} If $\mu>0$, we have that 
$\lim_{x\to+\infty} G(x)  =1/\mu.$

{\rm (ii)} If $\mu<0$, let $S^* = \sup_{i \ge 0} S_i <+\infty$ and assume that $H(x) = \bP(S^*\leq x)$ is subexponential (\textit{i.e.}\ $\overline{H*H}(x)\sim 2 \overline{H}(x)$ as $x\to +\infty$, where $\overline{H} = 1-H$). Then 
\[
G(x) \simx \frac{1}{|\mu|} \bP\big(  S^*  > x\big) \, .
 \]
If the integrated distribution function $I(x)= (1-\int_x^{+\infty} \bP(X_1>t) \dd t)_+$  is subexponential (as it is the case in \eqref{def:tail}), Veraverbecke's theorem \cite{cf:Ver, cf:Zach} gives that  $\bP(S^* >x) \sim I(x) / |\mu|$.
\end{theorem}

%

The case $\mu>0$ can be found in \cite[XI.9]{cf:Feller}.  For the case $\mu<0$, we were not aware of a reference (even if it must exist), so we prove it in Section~\ref{sec:finitemeanrenewal}, via elementary methods.

\section{Further discussion and useful estimates in the case $\ga=1$}
\label{sec:alpha1}

In this section, we focus on the case $\ga=1$, and we discuss the subtleties that might arise. 
One of the main difficulty is that the recentering term $b_n$ is not \emph{homogeneous}: the \emph{recentered} walk
$S_n-b_n$ is a sum of i.i.d.\ \emph{recentered} random variables, but that recentering depends on $n$:
$S_n - b_n= \sum_{i=1}^n (X_i -\mu(a_n) )$.
Hence, we are not able to simplify the problem by studying a random walk with \emph{centered} increments, as it is customary when $\ga>1$.

We focus on the case when $|\mu|=+\infty$ for a moment, for the simplicity of exposition (analogous reasoning holds when $\mu=0$).
Recall the definition \eqref{def:ell} of $\ell(\cdot)$, and let us discuss the behavior of the centering constant $b_n$.

$\ast$  If $p>q$, then $b_n \sim (p-q) n \ell(a_n)$: we get that $b_n/a_n \to +\infty$, since $\ell(x)/L(x)\to+\infty$ and $a_n \sim n L(a_n)$. Hence $S_n/b_n$ converges in probability to $1$, and $S_n$ goes in probability to $+\infty$, even though Proposition~\ref{prop:drift} tells that $S_n$ does not drift to $+\infty$.

$\ast$ If $p=q$ then it is more tricky, and we can have a variety of possible behaviors: $b_n = o(a_n)$ (we may set $b_n \equiv 0$); $0< \limsup_{n\to+\infty} |b_n|/a_n <+\infty$;  $\lim_{n\to+\infty} b_n/a_n =+\infty$ (but still $b_n =o(n \ell(a_n))$); and it is not excluded that $ \limsup_{n\to+\infty} b_n/a_n =+\infty$ and $ \liminf_{n\to+\infty} b_n/a_n =+\infty$.

\subsection{About the transience/recurrence of $S_n$}
\label{sec:transient}

We recall that in the case of a finite mean, $S_n$ is transient if $\mu\neq 0$ (by the strong law of large numbers), and recurrent if $\mu=0$.
In the case $\ga=1$ with $\mu=+\infty$ (and $p,q\neq 0$), the random walk is shown not to drift neither to $+\infty$ nor $-\infty$ (Proposition~\ref{prop:drift}):
the central (and subtle) question is therefore to know whether the random walk is transient or recurrent. 
Let us  consider the Green function at~$0$, $G(0):=\sum_{n} \bP(S_n=0)$.

\subsubsection*{First case}
If  $\sup_n |b_n|/a_n <+\infty$, then $\bP(S_n=0)$ is of order $1/a_n$: by the local limit theorem \eqref{eq:LLT} $a_n \bP(S_n =0) = (1+o(1)) g(-b_n/a_n)$, and it is therefore bounded away from~$0$ and infinity. We conclude that the walk is transient if and only if $\sum (a_n)^{-1} <+\infty$, and Wei \cite{cf:Wei} gives another characterization.
\begin{proposition}
\label{prop:transient1}
If \eqref{def:tail} holds with $\ga=1$, and  
$\sup_{n} |b_n| / a_n <+\infty$ (necessarily $p=q$), then
\[ (S_n)_{n\ge 0} \text{ is transient }\quad \Longleftrightarrow \quad \sum_{n=1}^{+\infty} \frac{1}{a_n} <+\infty  \quad \Longleftrightarrow \quad  \sum_{n=1}^{+\infty} \frac{1}{n L(n)} <+\infty \, . \]
\end{proposition}

\subsubsection*{Second Case.} If $ \lim_{n\to+\infty} b_n/a_n \to +\infty$, we have that $S_n \to +\infty$ in probability, but we cannot conclude that $S_n$ is transient, in particular because Proposition~\ref{prop:drift} tells that $\liminf S_n = -\infty$ a.s.
We have $\bP(S_n=0) = \bP(S_n-b_n =-b_n)$ and Theorem \ref{thm:localLD} gives that is is bounded by a constant times $(a_n)^{-1} n \bP(X_1>b_n) \sim p(b_n)^{-1} L(b_n)/L(a_n)$. Hence, a sufficient condition for the walk to be transient is that 
\[\sum_{n\ge 1} (b_n^{-1}) L(b_n)/L(a_n) <+\infty.\]
However Theorem~\ref{thm:localLD} does not provide a lower bound for $\bP(S_n=0)$, so the question of the recurrence/transience cannot be settled. We now give a simple sufficient condition for the transience of the random walk.
 
\begin{proposition}
 \label{prop:transient2}
 Assume that \eqref{def:tail} holds with $\ga=1$, and assume that $|\mu|=+\infty$. If $p>q$ (we then have $b_n/a_n\to+\infty$) and if additionally \eqref{hyp:localtail2} holds, then $S_n$ is transient. 
 \end{proposition}

 Note that the local assumption \eqref{hyp:localtail2} we need is only on the left tail of $X_1$: since we already know that $S_n\to+\infty$ in probability when $p>q$, we simply need to control the (large) jumps to the left that might make the random walk visit $0$.
\begin{proof}
We use Theorem \ref{thm:localDoney} to get that there is a constant $C$ such that 
\[\bP(S_n = 0) = \bP(S_n - \lfloor b_n \rfloor = -\lfloor b_n \rfloor ) \le C n b_n^{-2}  L(b_n)  \, . \]
Then, to show that $(S_n)_{n\ge 1}$ is transient, and since $b_n= n\mu(a_n)\sim (p-q) n\ell(a_n) \sim (p-q) n\ell(b_n) $ (see Lemma~\ref{lem:ell} below), it is therefore sufficient to show that 
\[\sum_{n=1}^{+\infty} \frac{L(b_n)}{b_n \ell(b_n)} <+\infty \qquad \Longleftrightarrow \qquad   \sum_{k=1}^{+\infty} \frac{L(k)}{k \ell(k)^2} <+\infty \, .\]
The equivalence simply comes from a comparison of the sums with corresponding integral (we may work with differentiable slowly varying functions see \cite[Th.~1.8.2]{cf:BGT}), and a change of variable $k=b_n$, ${ d} k= (p-q)\ell(b_n) { d} n$.
Then Lemma~\ref{lem:sumell}-(i) below (with $f(t)=1/t^2$) shows the summability of the sum on the right-hand side, and concludes the proof.
\end{proof}

\subsubsection*{Other cases}
If for example $\limsup b_n/a_n  = +\infty$ and $\liminf b_n /a_n <+\infty$, it is even less clear, and it seems hopeless to conclude anything without further assumptions.

\subsection{Useful estimates on $\ell(\cdot), \ell^{\star}(\cdot)$}

We now collect a few estimates on the slowly varying function $\ell(\cdot),\ell^\star (\cdot)$ defined in \eqref{def:ell} (they are de Haan functions), which will be central in the proofs of Theorems \ref{thm:ladder} and \ref{thm:renewal}.

\begin{lemma}
\label{lem:ell}
Recall the definition \eqref{def:an} of $a_n$. We have that
\[ \ell(a_n) \sim \ell\big( n \ell(a_n) \big) \qquad \text{ as } n\to+\infty \, .\]
Similarly, we have $\ell^\star(a_n) \sim \ell^\star(n\ell^\star(a_n))$ as $n\to+\infty$.
\end{lemma}
As a consequence, when $\ga =1$ with $|\mu|=+\infty$ (resp.\ $\mu=0$) so that $b_n \sim (p-q) n\ell(a_n)$ (resp.\ $b_n \sim (q-p) n \ell^\star(a_n)$),  if $p \neq q$ we get that 
$\ell(a_n) \sim \ell(b_n)$  (resp.\ $\ell^\star(a_n) \sim \ell^\star(b_n)$).
The proof of Lemma~\ref{lem:ell} is not difficult and can be found in \cite{cf:AA}, so we omit it.

%

The next lemma deals with sums that appear naturally in the course of the proofs.
\begin{lemma}
\label{lem:sumell}
 Consider the case $\alpha=1$, and recall the definitions \eqref{def:ell} of $\ell(\cdot)$, $\ell^\star(\cdot)$.
 Let $f$ be a non-increasing function $\bbR_+^* \to \bbR_+$.
 \begin{enumerate}
 \item[(i)] If $|\mu|=+\infty$ then $\ell(n)\to+\infty$, and denoting $J=\int_1^{+\infty} f(t) dt$ we have as $n\to+\infty$
\[
\begin{split}
\text{if } J <+\infty,\quad \text{ then } & \quad  \sum_{k=n}^{+\infty} \frac{L(k)}{k } f\big(\ell(k) \big)  \sim \int_{\ell(n)}^{+\infty} f(t) dt \to 0  \, ;\\
\text{if } J =+\infty, \quad \text{ then } & \quad \sum_{k=1}^{n} \frac{L(k)}{k } f\big( \ell(k) \big)  \sim \int_1^{\ell(n)} f(t) dt \to +\infty  \, .
\end{split}
\]
\item[(ii)] If $|\mu|<+\infty$ then $\ell^\star(n)\to 0$, and denoting $K=\int_0^{1} f(t) dt$ we have as $n\to+\infty$
\[
\begin{split}
\text{if } K <+\infty,\quad \text{ then } & \quad  \sum_{k=n}^{+\infty} \frac{L(k)}{k } f\big(\ell^\star(k) \big)  \sim   \int_0^{\ell^{\star}(n)} f(t) dt \to 0 \, ;\\
\text{if } K =+\infty, \quad \text{ then } & \quad \sum_{k=1}^{n} \frac{L(k)}{k } f\big( \ell^\star(k) \big)  \sim  \int_{\ell^\star(n)}^1 f(t) dt \to +\infty  \, .
\end{split}
\]
 \end{enumerate}
We therefore have that $\sum_{k\geq 1} k^{-1} L(k) f(\ell(k))$ (resp.\ $\sum_{k\geq 1} k^{-1} L(k) f(\ell^\star(k))$) is convergent if and only if $\int_1^{+\infty} f(t) dt$ (resp.\ $\int_0^1 f(t) dt$) is convergent.
\end{lemma}

\begin{proof}
For (i), the asymptotic equivalence comes from a simple comparison of the sum with the following integral (since $k^{-1} L(k) f(\ell(k))$ is asymptotically non-increasing this is straightforward), which is computed explicitly thanks to a change of variable $t=\ell(u)$,  $dt = L(u)u^{-1} du$ (recall $\ell(u)\to +\infty$ as $u\to+\infty$):
\[\int_n^{+\infty}  \frac{L(u)}{ u}  f\big( \ell(u)  \big)  du =  \int_{\ell(n)}^{+\infty} f(t) dt  \  ; \qquad
\int_1^{n}  \frac{L(u)}{ u}  f\big( \ell(u) \big)  du =  \int_1^{\ell(n) } f(t) dt  . \]
For (ii), the result is proven in a similar manner, using the change of variable $t= \ell^\star(u)$,  $dt = - L(u)u^{-1} du$ (recall $\ell^\star(u)\to 0$ as $u\to+\infty$).
\end{proof}

\section{Fuk-Nagaev's inequalities and  local large deviations}
\label{sec:FukNagaev}

From now on, we use $c,C, c',C',...$ as generic (universal) constants, and we will keep the dependence on parameters when necessary, writing for example $c_{\gep},C_{\gep}$ for constants depending on a parameter $\gep$.

\subsection{Fuk-Nagaev inequalities}

Let us first recall known Fuk-Nagaev inequalities---they are collected for example in \cite{cf:Nag79}, and are based on standard Cram\'er-type exponential moment calculations. Recall that $M_n = \max_{1\le i\le n} \{X_i\}$.
\begin{theorem}
\label{thm:fuknag}
Assume that \eqref{def:tail} holds with $\ga\in(0,2)$.
There exist a  constant $c$ such that for any $y\leq x$, 
\begin{itemize}
\item[(i)] If $\ga<1$,
\[\bP\left( S_{n} \geq x ; M_n \leq y \right) \leq 
e^{\frac xy} \Big( 1+\frac{ c x }{  n y^{1-\ga} L(y) }  \Big)^{- \frac xy}  \leq \Big( e c^{-1}\,  \frac{ y}{x} \, n L(y) y^{-\ga}\Big)^{ \frac xy} \, . \]
\item[(ii)] If $\ga>1$,
\[\bP\left( S_{n} - \mu  n\geq x ; M_n \leq y \right) \leq 
 \Big( c \frac{  y}{x} \, n L(y) y^{-\ga} \Big)^{- \frac{x}{y}  }  \, .
\]
\item[(iii)] If $\ga=1$,  we have for any $y\leq x$,
\begin{align*}
\bP\left( S_{n}\geq x ; M_n \leq y \right) &\leq e^{\frac{x}{y}}\Big( 1+ \frac{c x}{ n L(y)}\Big)^{- \frac{x - n \mu(y)}{y}} \, .
\end{align*}
\end{itemize}
\end{theorem}

The case $\ga<1$ is in \cite[Thm 1.1]{cf:Nag79} (take $t=1$ so that $0\le A(1,Y) \le c' n y^{1-\ga} L(y)$).
The case $\ga=1$ comes from \cite[Thm 1.2]{cf:Nag79} (take $t=2$ so that $0\le A(2,Y) \le c' n y^{2-\ga} L(y)$). For $\ga>1$, \cite[Thm 1.2]{cf:Nag79} does not directly give item (ii), but 
\[\bP\left( S_{n} - \mu  n\geq x ; M_n \leq y \right) \leq 
e^{\frac xy}  \Big( 1+\frac{c x }{ n y^{1-\ga} L(y) }  \Big)^{- \frac{x}{y} + \frac {n(\mu(y)-\mu)}{y} }\, . \]
Then we can use that  $|\mu(y)-\mu| =O( y^{1-\ga} L(y))$, to get that 
\[\bP\left( S_{n} - \mu  n\geq x ; M_n \leq y \right) \leq 
e^{\frac xy}  \Big( 1+\frac{c x }{ n y^{1-\ga} L(y) }  \Big)^{- \frac{x}{y} \big(1 -  c' n y^{1-\ga} L(y) /x\big)}\, . \]
Then, there is some $c''$ such that $ (1+c u)^{1-c'/u} \ge c'' u$ for all $u\ge 0$, so that we get Theorem~\ref{thm:fuknag}-(ii).

All these results remain valid if we only assume an upper bound $\bP(|X_1|>x) \le c\, L(x) x^{-\ga}$. 
Also, the case $\ga=2$ (random walks in the Normal domain of attraction) can be dealt with, see Corollary 1.7 in \cite{cf:Nag79}.
For the case $\ga = 1$, Theorem~\ref{thm:fuknag} gives some bound, but it is not optimal in general.
However, if $X_1$ has a symmetric distribution,  we  have that $\mu(y)\equiv 0$ and  $b_n \equiv 0$, so Theorem \ref{thm:fuknag} yields immediately the inequality
\begin{equation}
\label{eq:symmetric}
\bP\left( S_{n} -b_n \geq x ; M_n \leq y \right) \leq e^{\frac xy }\Big( 1+ \frac{cx}{ n L(y)}\Big)^{- \frac xy  } \, .
\end{equation}
The general case needs more work: we need the following result for non-negative $X_1$.
\begin{proposition}
\label{thm:alpha1}
Assume that $X_1\ge 0$ and that \eqref{def:tail} holds with $\ga=1$.
There exists a constant $c>0$, and for every $\gep>0$ there is some $C_{\gep}>0$, such that for any $x \ge C_{\gep} a_n$ and any $y\le x$
\begin{equation}
\label{alpha1pos}
\begin{split}
\bP&\left( S_{n} - b_n \geq x ; M_n \leq y \right) \leq 
e^{\frac xy}\Big( 1+ \frac{c x}{ n L(y) } \Big)^{-(1-\gep) \frac xy} \, ,\\
\bP&\left( S_{n} - b_n \leq - x  \right) \leq \exp\big( - \big( x/a_n \big)^{1/\gep} \big)  \, .
\end{split}
\end{equation}
\end{proposition}
From Proposition~\ref{thm:alpha1}, we obtain Theorem~\ref{thm:FukNagalpha1}-\eqref{eq:fuknag1}, by separating the $X_i$'s into a positive and a negative part:
$X_i^+ := X_i \ind_{\{X_i>0 \}}$ and $X_i^- := - X_i\ind_{\{X_i<0 \}}$,
so that $X_i^+,X_i^-$ are non-negative.
Naturally, we also define $S_n^+ := \sum_{i=1}^n X_i^+$ and $S_n^{-} := \sum_{i=1}^n X_i^{-}$ so that $S_n = S_n^+ - S_n^-$; and also $b_n^+ := n \bE[X_1^+ \ind_{\{X_1^+ \le a_n\}}]$ and $b_n^{-}:=n \bE[X_1^- \ind_{\{X_1^- \le a_n\}}]$, so that $b_n^+$ (resp.\ $b_n^-$) is a centering sequence for $S_n^+$ (resp.\ $S_n^-$), and $b_n = b_n^+ - b_n^{-}$.
Then, 
\begin{align}
\notag
\bP\Big( S_n - b_n \ge x \, ; \, M_n \le y\Big) \le \bP \Big( S_n^+ - b_n^+ \ge (1-\gep) x &\, ;\, \max_{1\le i\le n} X_i^+ \le y \Big)\\
 &+ \bP\Big( S_n^- - b_n^- \le -  \gep x \Big) \, .
 \label{eq:neg-pos-parts}
\end{align}
Then, we may use \eqref{alpha1pos} for both terms, and we obtain Theorem~\ref{thm:FukNagalpha1}-\eqref{eq:fuknag1}, by possibly adjusting the constants (for instance to treat the case $x\ge a_n $ instead of $x\ge C_{\gep} a_n$).

Note that in \eqref{eq:fuknag1}, the large deviation may come from two different possibilities: either the positive part makes a few jumps of length $y$ (the number of such jumps is approximately $(1-\gep)x/y$), giving the first term; either the negative part makes a large deviation to lower its value, giving rise to the second term (which is not affected by the truncation $M_n\le y$).

\subsection{An easy consequence: Theorem~\ref{thm:largedev}}

We now prove Theorem~\ref{thm:largedev}, as a consequence of the above Fuk-Nagaev inequalities.
We write it only for large deviations to the right (\textit{i.e.}\ $x/a_n \to+\infty$), the other case being symmetric.
For any fixed $\gep>0$, we write
\begin{equation}
\label{eq:largedev1}
\bP(S_n -b_n >x)  = \bP(S_n - b_n > x ; M_n > (1-\gep) x) + \bP(S_n - b_n >x ; M_n \le (1-\gep) x) \, . 
\end{equation}

\smallskip

{\it First term.}
It gives the main contribution.
A lower bound is, by exchangeability and independence of the $X_i$'s
\begin{align*}
\bP \big( \exists  &\ i \, , X_i > (1+\gep)x, S_{n} - b_n >x \big) \\
&\ge n \bP \big(X_1> (1+\gep) x \big) \bP \big(S_{n-1} - b_n > -\gep x \big) - \bP\big( \exists\ i\neq j , X_i, X_j > (1+\gep) x \big) \\
& \ge (1-\gep) n \bP \big(X_1> (1+\gep) x \big) -  \binom{n}{2} \bP \big(X_1> (1+\gep) x \big)^2 \, .
\end{align*}
The last inequality holds for $n$ large enough: we used that $(S_{n-1} -b_n)/a_n$ converges in distribution and $x/a_n \to+\infty$. In the last line, the second term is negligible compared to the first one, since $n\bP \big(X_1> (1+\gep) x \big) \to 0 $ as $x/a_n \to +\infty$.

For an upper bound, we use simply a union bound to get
\[\bP \big( M_n > (1-\gep) x \big) \le n \bP \big( X_1 > (1-\gep) x \big) \, .\] 
Thanks to \eqref{def:tail}, the first term in \eqref{eq:largedev1} is therefore asymptotically bounded from below  by $(1-\gep) (1+\gep)^{-\ga} (p-\gep) n L(x) x^{-\ga}$ and from above by $(1-\gep)^{-\ga} (p+\gep) n L(x) x^{-\ga}$.

\smallskip
{\it Second term.}
It remains to prove that, for any arbitrary $\gep>0$, the second term in~\eqref{eq:largedev1} is $o(nL(x) x^{-\ga})$ as $x/a_n \to+\infty$.
We decompose again this probability into two parts. The first part is
\[  \bP\Big(S_n -b_n >x ; M_n \in  (x/8,(1-\gep)x] \Big) \le n \bP(X_1 \ge x/8)  \bP\big(S_{n-1}  -b_n > \gep x \big) = o(n L(x) x^{-\ga}) \, .  \]
The first inequality comes from the exchangeability and independence of the $X_i$'s, and the second one comes from the convergence in distribution of $(S_{n-1} - b_n)/a_n$ together with $x/a_n \to+\infty$.
The last part ($M_n\le x/8$) is controlled thanks to the above Fuk-Nagaev Theorems~\ref{thm:FukNagalpha1}-\ref{thm:fuknag}: we have that in any case, taking $\gep =1/2$ in \eqref{eq:fuknag1}, as $x/a_n\to +\infty$
\[ \bP\big(S_n -b_n >x ; M_n  \le x/8 \big) \le \Big( c n L(x) x^{-\ga}\Big)^{4} + e^{ - c (x/a_n)^2} = o\Big( \frac{L(x)}{L(a_n)} (x/a_n)^{-\alpha} \Big) \, , \]
where we used that $n \sim a_n^{-\ga} L(a_n)$ (so that the last term is indeed $o(n L(x) x^{-\alpha})$).

\smallskip
In conclusion, since $\gep>0$ is arbitrary, we get Theorem \ref{thm:largedev}.
\qed

\subsection{Proof of Proposition \ref{thm:alpha1}}

Recall that $X_1\geq 0$, and that $\mu(y)$ is non-decreasing.

\smallskip
\noindent
{\it Proof of the first part of \eqref{alpha1pos}.}
We start from Theorem \ref{thm:fuknag}: for any $y\le x'$ we have
\begin{equation*}
\bP\left( S_{n}\geq x' ; M_n\leq y \right) \leq \Big( 1+ c\frac{x'}{ n L(y)}\Big)^{- \frac{x' - n \mu(y)}{y} } \, .
\end{equation*}
Plugging $x'=b_n+x = n\mu(a_n) +x $ in this inequality, we get that for any $y\le x$
\begin{align}
\label{eq:Nagaev1}
\bP\left( S_{n}- b_n \geq x ; M_n \leq y \right) &\leq \Big( 1+ c \frac{n\mu(a_n) +x}{ n L(y)}\Big)^{- \frac{x}{y} + \frac{n(\mu(y) -\mu(a_n))}{y}} \, .
\end{align}
Then it is just a matter of comparing $n(\mu(y) - \mu(a_n))$ to $x$ (with $x\ge y$).

First, when $y\le a_n$, then $\mu(y) - \mu(a_n) \le 0$ since $X_1$ is non negative, so that
\[ \bP\left( S_{n}- b_n \geq x ; M_n \leq y \right) \le \Big( 1+  \frac{ c x}{ n L(y)}\Big)^{- \frac{x}{y} } \, . \]

When $y\ge a_n$, then we use the following claim (we prove it for $\ell(\cdot)$ defined in \eqref{def:ell}, but it obviously holds also for $\ell^\star(c\dot)$, and for  $\mu(\cdot)$ in the case of a non-negative $X_1$).
\begin{claim}
\label{claim:mu}
For every $\gd>0$, there is a constant $c_{\gd}$ s.t. for every $1\le u\le v$ we have
\begin{equation}
\label{claim1}
 \frac{1}{L(v)} \big( \ell(v) - \ell(u) \big)  \le c_{\gd} (u/v)^{\gd}   \, .
\end{equation}
Moreover, considering two sequences $(u_n), (v_n) \to+\infty$, if there is a constant $c>0$ such that $u_n \le v_n\le c u_n$ for all $n$ then we have
\begin{equation}
\label{claim2}
 \frac{1}{L(u_n)} \big( \ell(v_n) - \ell(u_n) \big) \simn \log\big( v_n /u_n\big)
\end{equation}
\end{claim}
\noindent
\textit{Proof of the Claim.}
We write
\begin{align*}
 \frac{\ell(u) - \ell(v)}{L(v)}  = \int_{u}^{v} \frac{L(t) t^{-1}}{L(v)} {\rm dt} \le c_{\gd} \Big( \frac{u}{v} \Big)^{\gd/2} \int_{u}^{v} \frac{ {\rm d} t }{t}
\end{align*}
where we used Potter's bound (see \cite{cf:BGT}) to get that there is a constant $c_{\gd}$ such that uniformly for $t \ge v$ we have $L(t)/L(v) \le c_{\gd} (t/v)^{\gd/2}$. Then, the last integral is equal to $\log(u/v) \le c_{\gd} (u/v)^{\gd/2} $ so the first part of the claim is proven. 
For the second part, this is standard and comes from the same computation, together with the fact that  $L(t)/L(v_n) \to 1$ uniformly for $t\in [u_n, v_n] \subset [u_n,cu_n]$:
\[\frac{\ell(v_n) - \ell(u_n)}{L(u_n)}  = \int_{u_n}^{v_n} \frac{L(t) t^{-1}}{ L(v_n)} {\rm d} t = (1+o(1))\int_{u_n}^{v_n} \frac{{\rm d} t }{t}  \simn \log (v_n/u_n) \, .\]
\qed

From Claim~\ref{claim:mu} (take $\gd=1/2$) we get that, for any $y\ge a_n$, and since $x\ge y$
\[ n(\mu(y) - \mu(a_n)) \le c_{\gd} n L(a_n) (y/a_n)^{1/2} \le c \frac{n L(a_n)}{a_n} (x/a_n)^{-1/2} \times x .\]
Therefore, by choosing $C_{\gep}>0$ large enough, we get that $ n(\mu(y) - \mu(a_n)) \le \gep x$ provided that $x\ge C_{\gep} a_n$, $a_n\le y \le x$.
Plugged in \eqref{eq:Nagaev1}, we get
\[ \bP\left( S_{n}- b_n \geq x ; M_n \leq y \right) \leq e^{\frac xy}\Big( 1+ c \frac{n\mu(a_n) +x}{ n L(y)}\Big)^{- (1-\gep) \frac{x}{y} } \le  e^{\frac xy}\Big( 1+  \frac{c x}{ n L(y)}\Big)^{- (1-\gep) \frac{x}{y} } \, .\]
Hence, the first part of \eqref{alpha1pos} is proven.

\smallskip
\noindent
{\it Proof of the second part of \eqref{alpha1pos}.}
We write, for any $t>0$ and any $a_n \le x \le b_n$ (recall that $S_n-b_n \ge -b_n$)
\[ \bP \left(S_n -b_n \le -x \right) \le e^{-t (x-b_n)} \bE[e^{-t X_1}]^n  \, .\]
We also use that, because $X_1 \ge 0$ and thanks to \eqref{def:tail}, we have that there is a constant $c>0$ such that for any $t\le 1$
\begin{equation}
\label{Laplace:smallt}
\bE[e^{-t X_1}] - 1 + t \mu(1/t)  \le  c  t L(1/t) \, .
\end{equation}
Indeed, one simply writes that the absolute value of the left hand side is 
\begin{align*}
\bigg|  \sum_{n=1}^{1/t} \Big( 1-e^{-tn} -tn \Big)& \bP(X_1 =n)+ \sum_{n>1/t} (1-e^{-tn}) \bP(X_1 =n) \bigg| \\
&\le \sum_{n=1}^{1/t} t^2 n^2 \bP(X_1 =n) + \sum_{n>1/t}  \bP(X_1 =n) \\
& = t^2 \bE\big[ (X_1)^2 \ind_{\{X_1\le 1/t\}}\big] + \bP(X_1  >1/t) \, .
\end{align*}
Then, one easily get that  both terms are $O(t L(1/t))$, using \eqref{def:tail}.

Thanks to \eqref{Laplace:smallt}, and using that $1+x\le e^{x}$, we get that for any $t\le 1$
\begin{align}
\bP \left( S_n -b_n \le -x \right) & \le \exp \Big( -t (x-b_n) - n t \mu(1/t) + c n t L(1/t) \Big) \notag\\
& = \exp\Big( - t x  + nt \Big[ \mu(a_n)- \mu(1/t) + c  L(1/t) \Big] \Big)\, .
\label{eq:almostthere}
\end{align}
Then, we fix $\gep>0$ and choose $t:= (a_n)^{-1} \times (x/a_n)^{(1-\gep)/\gep}$.  Since $a_n \le x\le b_n$, we have that $ 1\le (x/a_n)^{(1-\gep)/\gep} \le (b_n/a_n)^{(1-\gep)/\gep}$, so that we indeed have $t\le 1$, and also $1/t\le a_n$.  Thanks to Claim~\ref{claim:mu}, we get that there is a constant $c_{\gep}$ such that $\mu(a_n) - \mu(1/t) \le c_{\gep} (t a_n)^{\gep} L(a_n)$, and Potter's bound  also gives that $L(1/t)  \le c_{\gep} (t a_n)^{\gep} L(a_n)$.
We therefore get that the r.h.s.\ of \eqref{eq:almostthere} is bounded by
\begin{equation}
 \exp\Big( - t x  + (1+c_{\gep}) n t  (t a_n)^{\gep} L(a_n) \Big)
 \le  \exp\Big( - \Big(\frac{x}{a_n} \Big)^{1/\gep}  + c'_{\gep}\Big(\frac{x}{a_n} \Big)^{(1-\gep^2)/\gep}  \Big)
\end{equation}
where we used the definition of $t$, together with the fact that $n L(a_n) \sim a_n$ for the second term in the exponential. Hence, there exists some $C_\gep>0$ such that provided that $x/a_n\ge C_{\gep}$ we have
\[\bP \left( S_n -b_n \le -x \right) \le \exp\Big( -\frac12 \big(x /a_n \big)^{1/\gep}   \Big) \, ,\]
which ends the proof of the second part of \eqref{alpha1pos}, the factor $1/2$ being irrelevant.

\section{Local large deviations}
\label{sec:locallarge}

\subsection{Local versions of Fuk-Nagaev inequalities}
The following Proposition proves Theorem~\ref{thm:FukNagalpha1}-\eqref{eq:localLD}, and together with a local version of Theorem~\ref{thm:fuknag}.
\begin{proposition}
\label{thm:localfuknag}
Assume that \eqref{def:tail} holds.
There is a  constant $C>0$ such that for any $\gep>0$, there is some $C_{\gep}$ such that: for $x\ge C_{\gep} a_n$ in the case $\ga=1$ and for $x\ge C_{\gep}$ if $\ga\neq 1$, for any $y\le x$
\[ \bP\left( S_{n} -  \lfloor b_n \rfloor = x ; M_n \leq y \right) \leq \frac{C}{a_n}
\bP\Big( S_{\lfloor  n/2 \rfloor} -  b_{\lfloor n/2 \rfloor}  \ge (1-\gep) \frac{x}{2}  ; M_{\lfloor n/2 \rfloor } \le y \Big)  \,  . \]
An upper bound is then given by Theorem~\ref{thm:FukNagalpha1}-\eqref{eq:fuknag1} if $\ga=1$, or  by Theorem~\ref{thm:fuknag} if $\ga\neq 1$.
\end{proposition}
This result is similar to \cite[Thm.~1.1]{cf:CD}, but here the estimate holds even when $y\ll x$ which is not the case in \cite{cf:CD}.
In view of Theorems~\ref{thm:FukNagalpha1}, one obtains the upper bound $(cn y^{1-\ga} L(y)/x )^{(1-\gep)^2 x/2y}$ if $\ga = 1$ for $x\ge C_{\gep} a_n$ (the case $\ga\neq 1$ is analogous): we might be able to improve the exponent to $\lceil x/y \rceil$ as in \cite[Thm.~1.1]{cf:CD} (at least when $y$ is a constant times~$x$), but we do not pursue this level of optimality here.

\begin{proof}
Let us denote $\hat S_n := S_n - \lfloor b_n \rfloor$ the ``recentered'' walk.
We decompose $\bP(\hat S_n = x)$ according to whether $S_{\lfloor  n/2 \rfloor} - \tfrac12 \lfloor b_n\rfloor \ge x/2$ or not, so that we obtain
\begin{align*}
\bP \big( \hat S_n = x  ; M_n \le y \big) \le &\bP\Big( \hat S_n = x  ; S_{\lfloor  n/2 \rfloor} - \tfrac12 \lfloor b_n \rfloor \ge x/2   \, ;\, M_{\lfloor  n/2 \rfloor} \le y \Big) \\
& + \bP\Big( \hat S_n = x  ;  S_n - S_{\lfloor  n/2 \rfloor} - \tfrac12 \lfloor b_n \rfloor \ge x/2 \, ; \, \max_{\lfloor n/2 \rfloor < i\le n} X_i \le y \Big)\, .
\end{align*}
The two terms are treated similarly, so we only focus on the first one.
We have
\begin{align}
\bP\Big(  \hat S_n = x   \,  & ;\, S_{\lfloor  n/2 \rfloor} - \tfrac12 \lfloor b_n \rfloor \ge x/2  ; M_{\lfloor n/2 \rfloor} \le y  \Big) \notag\\
&= \sum_{z \ge \tfrac12 \lfloor b_n \rfloor + x/2} \bP\big( S_{\lfloor n/2\rfloor} = z  ; M_{\lfloor n/2 \rfloor } \le y\big) \bP \big( S_n - S_{\lfloor n/2\rfloor}  =  \lfloor b_n \rfloor+ x-z \big)  \notag\\
& \le \frac{C}{a_{n}}  \sum_{z \ge \tfrac12 \lfloor b_n \rfloor + x/2} \bP\big( S_{\lfloor n/2\rfloor} =z ; M_{\lfloor n/2 \rfloor } \le y\big)
 \notag  \\
 &=  \frac{C}{a_{n}}  \bP\Big( S_{\lfloor  n/2 \rfloor} - \tfrac12 \lfloor b_n \rfloor \ge x/2  ; M_{\lfloor  n/2 \rfloor} \le y  \Big) \, ,
\label{firstpiece}
\end{align}
where we used Gnedenko's local limit theorem  \eqref{eq:LLT} to get that there is a constant $C>0$ such that for any $k\ge 1$ and $y \in \bbZ$, we have $\bP(S_k = y) \le C /a_k $.

It remains to control $\tfrac12 \lfloor b_n \rfloor  - b_{\lfloor n/2 \rfloor}$. When $\ga \in(0,1)$ we have that $b_n\equiv 0$ so this quantity is equal to $0$, and when $\ga>1$ we have $b_k = k \mu$ in which case we get $\tfrac12 \lfloor b_n \rfloor  -  \lfloor b_{n/2}\rfloor  \ge - |\mu| $.
When $\ga=1$, this is more delicate but not too hard:
\begin{align*}
\tfrac12 n \mu(a_n) - \lfloor n/2 \rfloor \mu(a_{\lfloor n/2 \rfloor}) &\ge \frac n2 \Big[ \mu(a_n) - \mu(a_{\lfloor n/2\rfloor}) \Big] - |\mu(a_{\lfloor n/2\rfloor }) | \\
&\ge  - c_0\,  n L(a_n)  -  |\mu(a_{\lfloor n/2\rfloor })| \ge - 2 c_0 a_n\, .
\end{align*}
For the second inequality we used  Claim \ref{claim:mu} (separating the positive and negative parts of $X_1$, using also that $a_n/a_{\lfloor n/2\rfloor}$ is bounded  above by a constant), and for the last inequality we used the definition of $a_n$ (and the fact that $|\mu(a_{\lfloor n/2\rfloor })| =o( a_n)$).
Therefore, provided that $x \ge C_{\gep} a_n$  with some constant $C_{\gep}$ large enough in the case $\ga=1$ (if $\ga\neq 1$, having $x\ge C_{\gep}$ is enough), we get that $\tfrac12 \lfloor b_n \rfloor  - b_{\lfloor n/2 \rfloor} \ge - \gep x/2 $. This concludes the proof.
%
\end{proof} 
 
%
%

\subsection{Improved local large deviations: proof of Theorem \ref{thm:localDoney}}

\label{sec:localdev}

We only consider large deviations to the right, i.e.\ $x\ge a_n$, since the other case is symmetric.
We give the proof of \eqref{eq:localLdoney} and \eqref{eq:localLdoney2} together, the latter using the same estimates.
We fix $\gep>0$ (we take $\gep=1/8$ when we prove \eqref{eq:localLdoney}, and we will choose $\gep$ arbitrarily small when we prove \eqref{eq:localLdoney2}),  and we write (recall $\hat S_n = S_n -\lfloor b_n \rfloor$)
\begin{align}
\bP \big( \hat S_n   =x \big) = \bP&\big( \hat S_n  =x, M_n \ge (1-\gep)x \big) \notag\\
&+ \bP\big( \hat S_n  =x, M_n \in (\gep x, (1-\gep) x )\big) +\bP\big( \hat S_n  =x, M_n \le \gep x \big)\, .
\label{threeterms}
\end{align}

The first term in \eqref{threeterms} is also the main one: by exchangeability of the $X_i$'s, we get that 
\begin{align}
\bP\big( \hat S_n  =x, M_n \ge & (1-\gep)x \big) = \sum_{y\ge (1-\gep) x} \bP( \hat S_n  = x , M_n =y) \notag \\
&\le  \sum_{y\ge (1-\gep)x} n \bP(X_1=y) \bP(S_{n-1} - \lfloor b_n\rfloor = x-y) \, .
\end{align}
Bounding $\bP(X_1 =y)$ by $\sup_{y\ge (1-\gep) x} \bP(X_1=y)$ and summing over $y$ the last probability, we therefore get
\[\bP\big( \hat S_n  =x, M_n \ge  (1-\gep)x \big)
 \le n \sup_{y\ge (1-\gep) x} \bP(X_1=y) \, .\]
Then, if we assume \eqref{hyp:localtail1}, we obtain
\begin{equation}
\bP \big( \hat S_n =x, M_n \ge (1-\gep)x \big) \le C n L(x) x^{-(1+\ga)} \, .
 \label{term1}
\end{equation}
If we assume \eqref{hyp:doney}, for any $\gep>0$ we have $\sup_{y\ge (1-\gep) x} \bP(X_1=y) \le (p+3\gep) L(x)x^{-(1+\ga)}$, provided that $x$ is large enough: we may replace the constant $C$ in \eqref{term1} by $(p+3\gep)$.

For the second term in \eqref{threeterms}, we  have:
\begin{align}
\bP\big( \hat S_n =x,  M_n \in & (\gep x, (1-\gep) x ) \big) = \sum_{y = \lfloor \gep x \rfloor +1}^{ \lceil (1-\gep) x \rceil -1} \bP( \hat S_n = x , M_n =y) \notag\\
& \le  \sum_{y = \lfloor \gep x \rfloor +1}^{ \lceil (1-\gep) x \rceil -1}  n \bP( X_1=y) \bP( S_{n-1}-\lfloor b_n \rfloor = x -y) \notag\\
& \le C_{\gep}\,  n L(x) x^{-(1+\ga)}\,  \bP(S_{n-1} - \lfloor b_n \rfloor \ge \gep x)\, ,
\label{term2}
\end{align}
where we used \eqref{hyp:localtail1} in the last inequality.
Moreover,  since $(S_n-b_n)/a_n$ converges in distribution, we get that $\bP(S_{n-1} - \lfloor b_n \rfloor \ge \gep x)\to 0$ if $x/a_n \to+\infty$, so  the second term is $o(n L(x) x^{-(1+\ga)})$.

For the last term in \eqref{threeterms}, we decompose it into two parts,
\[\bP\big( \hat S_n =x, M_n \le \gep x \big) \le \bP\big( \hat S_n=x, M_n \le   c a_n  \big) +   \bP\big( \hat S_n =x, M_n \in ( c a_n , \gep x)  \big) \]
The first part is controlled thanks to the local Fuk-Nagaev inequalities Theorems~\ref{thm:FukNagalpha1}-\eqref{eq:localLD} (or Proposition~\ref{thm:localfuknag} if $\ga\neq 1$): using that $x\ge a_n$ and that $nL(a_n) a_n^{-\ga} \to 1$, we get that 
\begin{equation}
\bP( \hat S_n =x, M_n \le  c a_n) \le \frac{C}{a_n} \Big(  \Big( c' \frac{x}{a_n} \Big)^{- c' x/a_n} +  e^{- c_2 (x/a_n)^2} \Big) \le \frac{c}{x}  \times  e^{- c'' x/a_n},
\label{term3-1}
\end{equation}
 which is negligible compared to \eqref{term1} (or \eqref{eq:localLdoney}) as $x/a_n \to+\infty$.
For the second part, we write
\begin{align}
\bP\big( \hat S_n &  =x, M_n \in[ c a_n,  \gep x ) \big)  = \sum_{j = \log_2 (1/\gep)}^{\log_2 (  c x/a_n ) -1} \bP \Big( \hat S_n  =x, M_n \in [2^{-(j+1)}, 2^{-j}) x \Big) \notag\\
& \le  \sum_{j =\log_2 (1/\gep)}^{\log_2 (c x/a_n ) } \sum_{y\in [2^{-(j+1)}, 2^{-j}) x} n \bP \big( X_1= y \big) \bP \big( S_{n-1} - \lfloor b_n \rfloor =x- y, M_{n-1} \le y \big) \notag\\
& \le C  \sum_{j=\log_2 (1/\gep)}^{\log_2 (c x/a_n ) } n L(2^{-j} x) (2^{-j} x)^{-(1+\alpha)} \bP\Big( S_{n-1} - \lfloor b_n \rfloor \ge x/2 , M_n \le 2^{-j} x \Big)\, ,
\label{decomp:Mn}
\end{align}
where we used \eqref{hyp:localtail1} to bound $\bP(X_1=y)$ uniformly for $y\in [2^{-(j+1)},2^{-j} ) x$.
Then, we use Fuk-Nagaev's inequalities Theorem \ref{thm:fuknag}-Theorem \ref{thm:alpha1}---leave aside the case $\ga=1$ for the moment---to get that (replacing $S_{n-1} - \lfloor b_n \rfloor$ by $S_n-b_n$ for simplicity)
\begin{align*}
\bP\Big( S_n -b_n \ge x/2 , M_n \le 2^{-j} x \Big) &  \le \Big( \frac{c 2^{j} }{n  L(2^{-j} x) (2^{-j } x)^{-\ga} } \Big)^{- 2^{j-2}}  \le \big(  c 2^j \big)^{ - 2^{j-2}} ,
\end{align*}
where we used that $2^{-j} x \ge  a_n$ for the range considered, so $n \bP(X_1 > 2^{-j } x) \le n\bP(X_1>a_n)$ and is bounded  above by a universal constant.
Plugged in \eqref{decomp:Mn}, and using Potter's bound to get that $L(2^{-j} x) \le c 2^j L(x) $ for all $j\ge 1$,
we therefore get that 
\begin{align}
\bP\Big( \hat S_n  =x, M_n \in[ c a_n,  \gep x ) \Big) & \le C n L(x) x^{-(1+\ga)} \sum_{j=\log_2 (1/\gep)}^{\log_2 ( c x/a_n)} 2^{(2+\ga) j} \big(  c 2^j \big)^{ - 2^{j-2}}  \notag \\
& \le c_{\gep} n L(x) x^{-(1+\ga)} , 
\label{term3-2}
\end{align}
where the constant $c_{\gep} $ can be made arbitrarily small by choosing $\gep$ small.
In the case $\ga=1$, Theorem \ref{thm:FukNagalpha1} gives an additional $e^{- c_2 (x/a_n)^2}$ in bounding $\bP\big( S_n -b_n \ge x/2 , M_n \le 2^{-j} x \big)$ for any $j\le \log_2( c x/a_n)$. Hence in \eqref{term3-2} we obtain an additionnal
\[ C n L(x) x^{-(1+\ga)} \sum_{j=1}^{\log_2(c x/a_n)} 2^{(2+\ga) j} e^{- c_2 (x/a_n)^2}\le C n L(x) x^{-(1+\ga)}  \times \Big(\frac{x}{a_n} \Big)^{3+\ga} e^{- c_2 (x/a_n)^2} ,  \]
which is $o(n L(x) x^{-(1+\ga)})$ as $x/a_n \to+\infty$.

\smallskip
In conclusion, combining \eqref{term1}-\eqref{term2}-\eqref{term3-1}-\eqref{term3-2}, we proved that fixing $\gep=1/8$ we get \eqref{eq:localLdoney}.
Assuming additionally \eqref{hyp:doney}, and in view of the remark made after   \eqref{term1}, we obtain that for any $\eta>0$ we can find $\gep>0$ (sufficiently small) such that, if $n$ and $x/a_n$ are large enough,
\[ \bP \big( S_n -b_n =x\big) \le (p+\eta) n L(x) x^{-(1+\ga)} \, .\]
This proves the upper bound part in Theorem~\ref{thm:localDoney}.

To get the lower bound in \eqref{eq:localLdoney2}, assume that $p>0$ (otherwise there is nothing to prove), and write
\begin{align*}
\bP \big( \hat S_n  =x\big) &\ge \bP\Big(\exists \, i \text{ s.t. } X_i \in \big( (1-\gep )x, (1+\gep) x \big) \, ; \, \forall j\neq i\ X_j \le x/2  \, ; \, \hat S_n = x\Big)\\
& = \sum_{y = \lceil (1-\gep)x \rceil }^{ \lfloor (1+\gep)x \rfloor} n \bP(X_1=y) \bP\big(S_{n-1} - \lfloor b_n \rfloor = x-y ; M_{n-1} \le x/2\big) \\
& \ge (1-3\gep) n p L(x) x^{-(1+\ga)} \bP\big(  S_{n-1} - b_n \in [-\gep x, \gep x] ; M_{n-1} \le x/2 \big).
\end{align*}
We used that $ \bP(X_1=y) \ge (1-3\gep) p L(x) x^{-(1+\ga)}$ uniformly for $y\in\big( (1-\gep )x, (1+\gep) x \big) $ and provided that $x$ is large enough,  because of~\eqref{hyp:doney}.
Then, the last probability converges to~$1$ as $n\to +\infty$  because $(S_{n-1} - \lfloor b_n \rfloor )/a_n$ and $M_{n-1}/a_n$ both converge in distribution, and $x/a_n\to +\infty$. Hence we have that for any $\eta>0$, we can find $\gep>0$ (sufficiently small) such that if $n$ and $x/a_n$ are large enough,
\[ \bP \big( S_n - \lfloor b_n \rfloor =x\big) \ge (1-\eta)p n  L(x) x^{-(1+\ga)} \, ,\]
which concludes the proof.\qed

\section{Ladder epochs: proof of Theorems \ref{thm:ladder}-\ref{thm:laddermu0}}
\label{sec:ladder}

To prove Theorem \ref{thm:ladder}, a crucial identity follows from the Wiener-Hopf factorization  (see e.g.\  Theorem 4 in \cite[XII.7]{cf:Feller}). Set $p_k := \bP(T_- >k)$ for every $k\ge 0$: for any $s\in[0,1)$
\begin{equation}
\label{eq:Feller}
p(s):= \sum_{k=0}^{+\infty} p_k s^k = \exp\Big(  \sum_{m=1}^{+\infty} \frac{s^m}{m} \bP(S_m \geq 0) \Big)\, .
\end{equation}

We present the proof in the case $\ga=1$ with infinite mean, \textit{i.e.}\ Theorem~\ref{thm:ladder} (it captures all the ideas needed), and then we adapt the proof to the case $\mu=0$ (\textit{i.e.}\ Theorem~\ref{thm:laddermu0}) in Section~\ref{sec:mu=0}.

\subsection{Preliminaries}
Let us first give the following lemma, which is the core of our proofs.

\begin{lemma}
\label{lem:P<0}
Assume that \eqref{def:tail} holds, with $\ga=1$ and $|\mu|=+\infty$. Recall the definition \eqref{def:ell} of $\ell(\cdot)$ and  \eqref{def:bn} of $b_n$.
If $p>q$, then $b_n\sim (p-q) n \ell(a_n) \to +\infty$ and
\[\bP(S_n < 0) \sim  \frac{q}{p-q} \frac{L( b_n )}{\ell(b_n)} \quad \text{as } n\to+\infty .\]
Moreover we have that
\[\sum_{k=1}^{n} \frac{1}{k} \bP(S_k<0) \sim \frac{q}{p-q} \log \ell(b_n) \quad \text{as } n\to+\infty\, .\]
If $q=0$, we interpret this as $o \big( \log \ell(b_n) \big )$.
The case $p<q$ is symmetric.

\noindent
If $p=q =1/2$, then $b_n = o(n \ell(a_n))$ but if $b_n/a_n \to +\infty$  we  have
\[\bP(S_n < 0) \simn   \frac{ n L(b_n)}{2  b_n} \, , \qquad \sum_{k=1}^{+\infty} k^{-1} \bP(S_k<0) =+\infty \, . \]
\end{lemma}

\begin{proof}
First, when $p>q$, we get thanks to Theorem \ref{thm:largedev} (since $b_n/a_n \to +\infty$) that
\begin{equation}
\label{eq:<0}
\bP(S_n < 0) = \bP(S_n - b_n < -b_n) \sim n q L(b_n) b_n^{-1} \sim \frac{q}{p-q} \, \frac{L(b_n)}{ \ell(b_n)} \quad \text{as } n\to+\infty \, .
\end{equation}
We used that $b_n  = n \mu(a_n)$ with $\mu(a_n)\sim (p-q) \ell(a_n) \sim (p-q) \ell(b_n)$ for the last part (see Lemma~\ref{lem:ell}).
The first asymptotic equivalence remains true as soon as ${b_n/a_n \to +\infty}$.

Then, it remains to estimate the sum $\sum_{k=1}^{n} k^{-1}\bP(S_k<0)$ or, because of \eqref{eq:<0}, of $ q \sum_{k=1}^n L(b_k)/b_k$.
By comparing with an integral, and using the function $\tilde b_t$ such  that $b_k\sim (p-q) \tilde b_k$ and $\partial_t \tilde b_t = \ell(\tilde b_t) $
(cf.~\eqref{def:btilde} below, we may assume that we work with differentiable function, see \cite[Thm.~1.8.2]{cf:BGT}), we obtain that 
\[\sum_{k=1}^n \frac{L(b_k)}{b_k}  \simn \int_{t=1}^n \frac{L(\tilde b_t)}{ \tilde b_t} \dd t = \frac{1}{p-q} \int_{u=1}^{\tilde b_n} \frac{L(u)}{u \ell(u)} \dd u \simn \frac{1}{p-q} \log \ell( b_n) \, .\]
We used a change of variables $u = \tilde b_t$, $\dd u = \ell(\tilde b_t) \dd t$, and then Lemma~\ref{lem:sumell}-(i) (with $f(t)=1/t$) for the last identity.

In the case where $p=q =1/2$ (so $b_k=o(k \ell(a_k))$) we use that according to \eqref{eq:<0} and provided that $b_n/a_n \to+\infty$, there is a constant $c>0$ such that 
\[k^{-1}\bP(S_k < 0) \ge c \frac{L(b_k)}{b_k}  \ge c' \frac{L(k \ell(a_k) )}{ k\ell(a_k)} \, .\]
And we  proved  just above that $\sum_{k=1}^{+\infty}\frac{L(k \ell(a_k) )}{ k\ell(a_k)} = +\infty $.
\end{proof}

A simple consequence of Lemma \ref{lem:P<0} is Proposition \ref{prop:drift} (in the case $|\mu|=+\infty$, the case $\mu=0$ being treated in Lemma~\ref{lem:P<0bis} below), thanks to \cite[XII.7~Thm.~2]{cf:Feller}.
Indeed, we get that $\sum_{k\ge 1} k^{-1} \bP(S_k <0) =+\infty$ as soon as $q\ne 0$: if $p>q$ or $p=q$ with $b_n/a_n\to +\infty$, this is directly  Lemma \ref{lem:P<0}; if $\sup_n  |b_n|/a_n <+\infty$, then this is just a consequence of the convergence in distribution of $(S_n-b_n)/a_n$ to get that $\bP(S_k <0) =\bP\big( (S_k-b_k)/a_k < -b_k/a_k\big)$ is uniformly bounded away from $0$, so that $\sum_{k\ge 1} k^{-1} \bP(S_k<0) =+\infty$; the general case when $p=q=1/2$ can be dealt with similarly, by observing as above that there is a constant $c$ such that $k^{-1}\bP(S_k<0) \ge c \frac{L(k \ell(a_k) )}{ k\ell(a_k)}$.

\subsection{The case $\lim_{n\to+\infty} b_n/a_n =b$}
\label{sec:pqbalanced}
This is case (i) in Theorem~\ref{thm:ladder}, which is standard, cf.\ Rogozin \cite{cf:Rog}. 
The sequence $\bP(S_k>0) =\bP( (S_k-b_k)/a_k > -b_k/a_k)$ converges to $\bP(Y> -b)$, where $Y$ is the limit in distribution of $(S_n-b_n)/a_n$, that is a symmetric Cauchy($1/2$) distribution ($p=q=1/2$), and $b=\lim_{n\to+\infty} b_n/a_n$.
Note that we could also characterize $b_n/a_n$ by $ \bE\big[\frac{X_1}{1+(X_1/a_n)^2} \big]$, see \cite[Thm.~8.3.1]{cf:BGT}.
We therefore get that 
\[ \lim_{n\to+\infty} \frac{1}{n} \sum_{k=1}^n \bP(S_k>0) = \bP(Y> -b) =:\rho \in(0,1),  \]
with $\rho = \frac12 +\frac{1}{\pi} \arctan ( 2b/\pi )$. This is Spitzer's condition:
from \cite[Thm.~8.9.12]{cf:BGT}, $T_-$ is in the domain of attraction of a positive stable random variable with index $\rho$, and (i) follows.

\subsection{The case $p<q$}
\label{sec:p<q}

We  first prove the weak result with the $o(1)$ in the exponent, and then turn to the precise statement under assumptions V1-V2.

\subsubsection*{General Case}
Denote
\begin{equation}
\label{def:f}
f(s) = \sum_{m=1}^{+\infty} \frac{s^m}{m} \bP(S_m\ge 0)\, ,
\qquad f'(s) = \sum_{m=0}^{+\infty} s^{m} \bP(S_{m+1}\ge 0)\, .
\end{equation}
We are able to obtain the behavior of $f(s)$ and $f'(s)$ as $s\uparrow1$. For $p<q$, Lemma \ref{lem:P<0} gives 
\begin{equation}
\sum_{k=1}^n k^{-1}\bP(S_k\ge 0) \simn \frac{p}{q-p} \log \ell(|b_n|)\, ; \quad \sum_{k=1}^n \bP(S_k\ge 0) \simn \frac{p}{q-p} \frac{n  L(|b_n|) }{\ell(|b_n|) }\, ,
\end{equation}
 (because $\bP(S_k \ge 0) \sim \frac{p}{q-p} L(|b_k|)/\ell(|b_k|)$). Therefore, Corollary 1.7.3 in \cite{cf:BGT} gives that 
 \begin{equation}
 \label{eq:ff'1}
 f(s) \sim  \frac{p}{q-p} \log \ell(|b_{1/(1-s)}|) \, ; \quad f'(s) \sim \frac{p}{q-p} \frac{1}{1-s} \frac{L(|b_{1/(1-s)}|)}{ \ell(|b_{1/(1-s)}|)}  \qquad \text{ as } s\uparrow 1 \, .
 \end{equation}

The identity  \eqref{eq:Feller} gives that $p(s) = e^{f(s)}$ for any $s\in[0,1)$ so that $p'(s) = f'(s) e^{f(s)}$: the estimates \eqref{eq:ff'1} allows us to derive that
\begin{equation}
\label{eq:p'}
p'(s) =  \frac{ L(|b_{1/(1-s)}|) }{1-s} \Big( \ell \big( |b_{1/(1-s)}| \big) \Big)^{ \frac{p}{q-p}-1+o(1)} \quad \text{as } s\uparrow 1 \, ,
\end{equation}
where the term $p/(q-p)$ has been absorbed in $\ell (|b_{1/(1-s)}|)^{o(1)}$
From this we would like to conclude that $\sum_{k=1}^n k p_k  =  n L(|b_n|) \big( \ell(|b_n|)\big)^{\frac{p}{q-p}-1 +o(1)}$, but we cannot directly apply Corollary~1.7.3 in \cite{cf:BGT} since we do not have a proper asymptotic equivalence. We therefore prove it directly.

\noindent
{\bf Upper bound.}
First, take $s=1-1/n$ in \eqref{eq:p'}, so that we get, as $n\to+\infty$
\begin{equation}
\sum_{k=0}^{+\infty} kp_k \Big(1-\frac{1}{n}\Big)^{k-1}  = n L(|b_n|) \big( \ell(|b_n|) \big)^{\frac{p}{q-p}-1 +o(1)}\, .
\end{equation}
Then, using that $p_k$ is non-increasing, we can write that
\[
\sum_{k=0}^{+\infty} kp_k \Big(1-\frac{1}{n}\Big)^{k-1} \ge \sum_{k=0}^{n} k p_n \Big(1-\frac{1}{n}\Big)^{n} \ge  c\,  n^2 p_n ,
\]
and we therefore get the upper bound
\begin{equation}
\label{eq:upperp}
p_n \le c \frac{1}{n} L(b_n) \big( \ell(b_n) \big)^{\frac{p}{q-p}-1 +o(1)}  = \frac{L(|b_n|) }{n} \big( \ell(|b_n|) \big)^{\frac{p}{q-p}-1 +o(1)} \, . 
\end{equation}

\noindent
{\bf Lower bound.}
The lower bound is a bit trickier. Let $\gep>0$, and define $t_n := \ell(|b_n|)^{\gep} \to +\infty$ as $n\to+\infty$. Setting $s= 1- 1/(nt_n) $ in \eqref{eq:p'}, we get that for $n$ sufficiently large
\begin{align}
\sum_{k=0}^{+\infty} k p_k \Big( 1- \frac{1}{n t_n}\Big)^{k-1} &=   n t_n L(|b_{n/t_n}|) \big(\ell(|b_{n/ t_n}|) \big)^{\frac{p}{q-p}-1 +o(1)} \notag\\
& \ge  n t_n^{1/2} L(|b_{n}|) \big(\ell(|b_{n}|) \big)^{\frac{p}{q-p}-1 +o(1)} .
\label{eq:lower1}
\end{align}
For the second inequality, we used that $|b_{n/t_n}| \ge t_n^{-2} |b_n|$ for $n$ large enough (since $b_k$ is regularly varying with index $-1$), and then
 Potter's bound to get that $L(t_n^{-2} | b_n| ) \le t_n^{-1/4} L(|b_n|)$ and $\ell(t_n^{-2} |b_n|)^{\frac{p}{q-p}-1 +o(1)} \le t_n^{-1/4}$ for $n$ large enough.

Now we may write, since $p_k$ is non-increasing,
\begin{equation}
\label{eq:twosums}
\sum_{k=0}^{+\infty} kp_k \Big(1-\frac{1}{n t_n}\Big)^{k-1} \le \sum_{k=0}^n k p_k + p_n \sum_{k=n+1}^{+\infty} k \Big(1-\frac{1}{nt_n} \Big)^{k-1} 
\end{equation}
For the first sum, we get thanks to \eqref{eq:upperp} that 
\[\sum_{k=0}^n k p_k  \le c \, n L(|b_n|) \big( \ell(|b_n|) \big)^{\frac{p}{q-p}-1 +o(1)}   = o(1)  n t_n^{1/2} L(|b_n|) \big( \ell(|b_n|) \big)^{\frac{p}{q-p}-1 +o(1)}  ,\]
where we used the definition of $t_n$, which is such that $ \ell(|b_n|)^{o(1)}= o(t_n^{1/2})$.
Now we have
$\sum_{k\ge 1}k  s^{k-1} = (1-s)^{-2}$, so that the second term in \eqref{eq:twosums} is bounded by  $ p_n (n t_n)^2$.
Hence, plugging \eqref{eq:twosums} (and the subsequent estimates) in \eqref{eq:lower1}, we obtain that
\[
n t_n^{1/2} L(|b_{n}|) \big(\ell(|b_{n}|) \big)^{\frac{p}{q-p}-1 +o(1)} \le o\Big(  n t_n^{1/2} L(b_n) \big( \ell(|b_n|) \big)^{\frac{p}{q-p}-1 +o(1)} \Big) +  p_n (nt_n)^2,
\]
so that we conclude that
\[
p_n \ge \frac{c}{n t_n^{3/2}} L(|b_n|) \big( \ell(|b_n|) \big)^{\frac{p}{q-p}-1 +o(1)}  \, .
\]
Recalling that $t_n =\ell(|b_n|)^{\gep}$, and since $\gep>0$ is arbitrary, we get that
\begin{equation}
p_n \ge \frac{L(|b_n|) }{n} \big( \ell(|b_n|) \big)^{\frac{p}{q-p}-1 +o(1)}  \, . 
\end{equation}

\subsubsection*{Under assumption V1-V2.}

Let us first introduce some notations. We construct $\tilde b_t$ an analytic function such that its derivative is given by $\ell(\tilde b_t)$.
Define
\begin{equation}
\label{def:btilde}
H(x) = \Big( \int_1^{t} \frac{{\rm d} x}{ \ell(x)} \Big)^{-1} \, , \quad \text{and } \quad \tilde b_t = H^{-1}(1/t)\, .
\end{equation}

Then, it is easy to verify that $H'(x) = -H(x)/\ell(x)$, so that  $\partial_t \tilde b_t =\ell(\tilde b_t)$ (using also that  $H(\tilde b_t) =1/t$).
Notice that we also have easily that  $H(x) \sim \ell(x)/x$ as $x\to+\infty$, so that $\tilde b_t \sim t \ell(\tilde b_t)$ as $t\to+\infty$. Thanks to Lemma \ref{lem:ell} we get that $\tilde b_t \sim t \ell(a_t)$ as $t\to+\infty$, and therefore $b_n \sim -(q-p) \tilde b_n$.
Let us define
\[g(s):= \log \ell\Big( \tilde b_{\frac{1}{1-s}} \Big) \, .\]
Using that $\partial_t \tilde b_t =\ell(\tilde b_t)$, we get that
\[g'(s) = \frac{1}{(1-s)^2} \frac{L\big( \tilde b_{\frac{1}{1-s}} \big)}{ \tilde b_{\frac{1}{1-s}} }  \stackrel{s\uparrow 1}\sim  \frac{1}{1-s} \frac{L\big( \tilde b_{\frac{1}{1-s}} \big)}{ \ell\big(\tilde b_{\frac{1}{1-s}} \big) } \, .\]
Since $L(\cdot)$ satisfies V1-V2, so does $ L\big( \tilde b_{\frac{1}{1-s}} \big) \big/ \ell\big(\tilde b_{\frac{1}{1-s}} \big)$, and we may apply Theorem~5 in \cite{cf:FO} to get that $g'(s) =\sum_{n=0}^{+\infty} a_n s^n$ with $a_n  \sim L(\tilde b_n) /\ell(\tilde b_n)$ as $n\to+\infty$. We end up with
\begin{equation}
g(s):= \log \ell\Big( \tilde b_{\frac{1}{1-s}} \Big) =\sum_{n=1}^{+\infty} g_n s^n \quad \text{with } g_n \sim \frac{L(\tilde b_n) }{n\ell(\tilde b_n)} \sim \frac{q-p}{p}  \frac{\bP(S_n \ge 0)}{n} \, .
\end{equation}
In view of \eqref{def:f}, we get that
\begin{equation}
f(s) = \frac{p}{q-p} g(s) + \sum_{n=1}^{+\infty} v_n s^n \quad \text{with } v_n = o \Big( \frac{L(\tilde b_n) }{n\ell(\tilde b_n)}\Big),
\end{equation}
so that
\begin{equation}
p'(s) = f'(s) \ell\big( \tilde b_{\frac{1}{1-s}} \big)^{\frac{p}{q-p}} \psi\big(\tfrac{1}{1-s} \big) \quad \text{with } \psi\big(\tfrac{1}{1-s} \big)  :=\exp\Big(\sum_{n=1}^{+\infty} v_n s^n  \Big) \, .
\end{equation}

\begin{lemma}
\label{lem:slowlyvar}
There exists some slowly varying function $\tilde L(\cdot)$ such that
\[\psi\big(\tfrac{1}{1-s} \big)  = \tilde L \Big( \ell\big( \tilde b_{\frac{1}{1-s}} \big) \Big)\, .\]
\end{lemma}
With this lemma in hand, we get that $p'(s)$ is regularly varying with index $-1$, 
\[p'(s) \sim \frac{L(\tilde b_{\frac{1}{1-s}})}{1-s} \ell\Big( \tilde b_{\frac{1}{1-s}} \Big)^{\frac{p}{q-p}-1} \tilde L \Big( \ell\Big( \tilde b_{\frac{1}{1-s}} \Big) \Big) \qquad \text{as } s\uparrow 1\, ,\]
so that by Corollary 1.7.3 in \cite{cf:BGT} we get that
\[ \sum_{k=1}^n k p_k \simn   n L\big( \tilde b_n \big)  \ell \big( \tilde b_n\big)^{\frac{p}{q-p}-1} \tilde L \big(  \ell \big( \tilde b_n\big) \big)  \, . \]
The result follows by using the monotonicity of $p_n$ (and the fact that $|b_n| \sim (q-p)\tilde b_n$).

\begin{proof}[Proof of Lemma \ref{lem:slowlyvar}]
Set $Q(t) = \ell( \tilde b_t)$ for simplicity, which is an increasing function. We want to show that $\tilde L(t):=\psi\big( Q^{-1}(t)\big)$ is slowly varying as $t\to+\infty$ ($t=(1-s)^{-1}$), \textit{i.e.}\ denoting $R(t)=Q^{-1}(t)$ for simplicity ($R(t)$ is increasing and $R(t)\to +\infty$ as $t\to+\infty$), we need to show that for any $c>1$,
\[\frac{\psi \big(R(ct) \big)}{\psi\big(R(t)\big)}  = \exp \left( \sum_{n=1}^{+\infty} v_n \bigg[ \Big( 1-\frac{1}{R(ct)}\Big)^n - \Big( 1-\frac{1}{R(t)}\Big)^n \bigg] \right) \to 1 \quad \text{as } t\to+\infty\, .\]
Since $v_n= o\big(L(\tilde b_n) /\tilde b_n\big)$, we write $v_n =\gep_n L(\tilde b_n) /\tilde b_n$ with $\gep_n \to 0$. In order to show that the sum in the exponential goes to $0$ as $t\to+\infty$, we split it into three parts.

{\it Part 1.}
For $n\le R(t)$ we use that
\[  \Big( 1-\frac{1}{R(ct)}\Big)^n - \Big( 1-\frac{1}{R(t)}\Big)^n\le  1- \Big( 1-\frac{n}{R(t)}\Big) = \frac{n}{R(t)}.\]
Hence the sum up to $n = R(t)$ is bounded by $ \sum_{n=1}^{R(t)} n v_n /R(t)$, and since $n v_n  \sim \gep_n  L(\tilde b_n) / \ell(\tilde b_n)$ goes to $0$ as $n\to+\infty$ (recall that $L(n)/\ell(n) \to 0$), we get that this first part goes to $0$ as $t\to +\infty$.

{\it Part 2.}
For $ R(t) < n \le R(2ct)$, we simply bound the sum by
\begin{equation}
\sum_{n= R(t)}^{R(2ct)} v_n \le \sup_{n \ge R(t)} \gep_n   \times \int_{R(t)}^{R(2ct)} \frac{L(\tilde b_u)}{\tilde b_u} {\rm d} u \, .
\end{equation}
Using that $Q(t) = \ell(\tilde b_t)$ we have that $Q'(t) = L(\tilde b_t) \ell(\tilde b_t) /\tilde b_t$, so that the integral is exactly $\big [   \log Q(s) \big]_{R(t)}^{R(2ct)} =\log 2c $ (recall that $R=Q^{-1}$). 
Since $\gep_n \to 0$, this second part also goes to $0$ as $t\to+\infty$.

{\it Part 3.}
For $n> R(2ct)$, we bound the sum by
\begin{align*}
\sum_{k=1}^{+\infty} &\sum_{n=R(2^k ct)}^{R(2^{k+1} ct)} \gep_n \frac{L(\tilde b_n)}{\tilde b_n} \Big( 1- \frac{1}{R(ct)}\Big)^{R(2^k ct)} \\
&\le \sup_{n \ge Q^{-1}(t)} \gep_n   \times \sum_{k=1}^{+\infty}  e^{ - R(2^k ct)/R(ct) } \int_{R(2^k ct)}^{R(2^{k+1} ct)} \frac{L(\tilde b_u)}{\tilde b_u} {\rm d} u \,.
\end{align*}
As above, the integral is equal to $\log 2$. Moreover, since $Q(\cdot)$ is slowly varying we get that for $t$ large enough (recall $R=Q^{-1}$), by Potter's bound, for any $k\ge 1$
\[Q \big( 2^{k} R(ct) \big) \le 2^k Q\big(R(c t) \big) = Q\big( R (2^k ct)\big)\, ,\]
giving that $2^{k} R(ct) \le R(2^k ct)$ since $Q$ is increasing. Hence, for $t$ large enough the third part is bounded by 
\[ \sup_{n \ge R(t)} \gep_n  \times \log 2 \sum_{k\ge 1} e^{-2^k}\, ,\]
which goes to $0$ as $t\to+\infty$.
\end{proof}

\subsection{The case $p>q$}
\label{sec:p>q}
This case is similar to the case $p<q$. We only prove the general case (with the $o(1)$ in the exponent), the improvement under assumption V1-V2 being identical to what is done above.

 Using the same definition of $f(s)$, and writing $\bP(S_m \ge 0) = 1- \bP(S_m < 0)$, we get that
\begin{equation}
\label{def:h}
f(s) = \sum_{m=1}^{+\infty} \frac{s^m}{m} \bP(S_m \ge 0) = \log\Big(\frac{1}{1-s} \Big) - h(s) \quad \text{with } h(s) = \sum_{m=1}^{+\infty} \frac{s^m}{m} \bP(S_m < 0)\, .
\end{equation}
Then, Lemma \ref{lem:P<0} gives that $\sum_{k=1}^{n} k^{-1} \bP(S_k < 0) \simn \frac{q}{p-q} \log \ell(b_n)$, and Corollary~1.7.3 in \cite{cf:BGT} gives that $h(s) \sim \frac{q}{p-q} \log \ell(b_{1/(1-s)})$ as $s \uparrow 1$.

Hence, we conclude thanks to \eqref{eq:Feller} that 
\[p(s) = \frac{1}{1-s} e^{-h(s)}  = \frac{1}{1-s} \Big( \ell(b_{1/(1-s)}) \Big)^{ - q/(p-q)+o(1)} \qquad \text{as } s\uparrow 1\, ,\]
and we  deduce from this the behavior of $p_n$ in the same way as above.

\noindent
{\bf Upper bound.}
Taking $s=1-1/n$, and using that $p_k$ is non-increasing, we get that
\[ n \big( \ell(b_{n}) \big)^{ - q/(p-q)+o(1)}  = \sum_{k=0}^{+\infty} p_k \Big( 1-\frac1n\Big)^k \ge  \sum_{k=0}^n p_n \Big( 1-\frac1n\Big)^k \ge c n p_n  ,\]
and therefore
\begin{equation}
p_n \le  c^{-1}\,  \big( \ell(b_{n}) \big)^{- q/(p-q)+o(1)} =   \big( \ell(b_{n}) \big)^{ -q/(p-q)+o(1)} \, .
\label{eq:upperp2}\, .
\end{equation}

\noindent
{\bf Lower bound.}
As above, we fix $\gep>0$ and define $t_n = \ell(b_n)^{\gep}$.
Taking $s=1-1/(nt_n)$, we get as in \eqref{eq:lower1} that for $n$ large enough
\begin{equation}
\label{eq:lower2}
 \sum_{k=0}^{+\infty} p_k \Big( 1-\frac{1}{nt_n}\Big)^{k} = nt_n \big( \ell(b_{nt_n}) \big)^{ -q/(p-q)+o(1)} \ge nt_n^{1/2} \big( \ell(b_{n}) \big)^{ q/(p-q)+o(1)}\, .
\end{equation}
On the other hand, as in \eqref{eq:twosums}, since $p_k$ is non-increasing we have
\begin{align*}
 \sum_{k=0}^{+\infty} p_k \Big( 1-\frac{1}{nt_n}\Big)^{k} & \le \sum_{k=0}^n p_k + p_n \sum_{k\ge n+1} \Big( 1-\frac{1}{nt_n}\Big)^{k} \\
 &= o\Big( nt_n^{1/2} \big( \ell(b_{n}) \big)^{ -q/(p-q)+o(1)} \Big) + nt_n p_n\, ,
\end{align*}
where we used \eqref{eq:upperp2} for the first term, together with the fact that $\ell(b_n)^{o(1)} =o(t_n^{1/2})$, and a standard computation for the second term.
Combining this with \eqref{eq:lower2} we get that
\[p_n \ge t_n^{-1/2} \big( \ell(b_{n}) \big)^{ -q/(p-q)+o(1)},\]
and since $t_n= \ell(b_n)^{\gep}$ with $\gep>0$ arbitrary, we get that
$p_n \ge   \ell(b_{n})^{ -q/(p-q)+o(1)}$.

\subsection{Further remarks on the case $p=q$}
\label{sec:p=q}

When $p=q =1/2$, then $b_n = o(n\ell(a_n))$. If $\lim_{n\to +\infty} b_n/a_n =b$, then Theorem \ref{thm:ladder} gives the correct asymptotic for $\bP(T_- >n)$.
In the case $\lim_{n\to +\infty} b_n /a_n =+\infty$ (the case where the limit is $-\infty$ is symmetric), then we still have as in \eqref{eq:<0} that
\begin{equation}
\bP(S_n < 0) \sim \frac12  n L(b_n) (b_n)^{-1} \quad \text{as } n\to+\infty.
\end{equation}
Since $b_n$ is regularly varying with exponent $-1$, we get that
\begin{equation}
\label{def:r}
r(n) := \sum_{k=1}^n \frac{1}{k} \bP(S_k<0)  \simn \sum_{k=1}^{n} \frac{L(b_n)}{ 2b_n}\, 
\end{equation}
is slowly varying.
Additionally, we get that $ r(n) = o(\log n) $ (since $\bP(S_n<0)\to 0$), and  also $r(n) \gg \log(\ell(a_n))$, in view of Lemma \ref{lem:P<0}, since $b_n = o(n \ell(a_n))$. We therefore have   
\[ \ell(a_n)^{1/o(1)} \le e^{r(n)} \le n^{o(1)} \, .\]
Then, the same scheme of proof as above gives the behavior of $p(s)$ and $p'(s)$ and in turns those of $p_n$: we get that
\begin{equation}
\label{eq:T-T+}
\bP(T_- >n ) =  e^{(1+o(1)) r(n)}  \quad \text{ and } \quad  \bP(T_+ >n) = \frac{L(b_n)}{ b_n} e^{(1+o(1)) r(n)}\, .
\end{equation}
Details are straightforward, and we do not develop further.
%
%

\subsection{Case $\ga=1$, $\mu=0$: proof of Theorem~\ref{thm:laddermu0}}
\label{sec:mu=0}

First of all, let us state the analogous of Lemma~\ref{lem:P<0} in the case $\ga=1$, $\mu=0$.
\begin{lemma}
\label{lem:P<0bis}
Assume that \eqref{def:tail} holds, with $\ga=1$ and $\mu=0$. Recall the definition \eqref{def:ell} of $\ell^\star(\cdot)$ and  \eqref{def:bn} of $b_n$.
 Then, if $p>q$, $b_n\sim  - (p-q) n \ell^\star(a_n) \to -\infty$ and
\[\bP(S_n > 0) \sim  \frac{p}{p-q} \frac{L( |b_n| )}{\ell^\star(|b_n|)} \quad \text{as } n\to +\infty\, .\]
Moreover we have that (recall $\ell^{\star}(n)\to 0$)
\[\sum_{k=1}^{n} \frac{1}{k} \bP(S_k > 0) \simn - \frac{p}{p-q} \log  \ell^\star(|b_n|) \quad \text{as } n\to +\infty \, .\]
The case $p<q$ is symmetric.
%
\end{lemma}

We skip the proof here since it is identical to that of Lemma~\ref{lem:P<0}, using Lemma~\ref{lem:sumell}-(ii) in place of  Lemma~\ref{lem:sumell}-(i).
We now turn to the proof of Theorem~\ref{thm:ladder2}, which is very similar to that of Theorem~\ref{thm:ladder}, the key identity being the Wiener-Hopf factorization \eqref{eq:Feller}.

\smallskip
\textit{(i) The case $p>q$.}
Denoting $f(s)$ as in \eqref{def:f}, Lemma~\ref{lem:P<0bis} and Corollary~1.7.3 in \cite{cf:BGT} give analogously to \eqref{eq:ff'1} 
 \begin{equation}
 \label{eq:ff'2}
 f(s) \sim  \frac{- p}{p-q} \log \ell^\star(|b_{1/(1-s)}|) \, , \quad f'(s) \sim  \frac{p}{p-q} \frac{1}{1-s} \frac{L(|b_{1/(1-s)}|)}{ \ell^\star(|b_{1/(1-s)}|)}  \qquad \text{ as } s\uparrow 1 \, .
 \end{equation}
Therefore,
\[p'(s) = f'(s) e^{f(s)} =  \frac{ L(|b_{1/(1-s)}|) }{1-s} \Big( \ell^\star\big (|b_{1/(1-s)}| \big) \Big)^{ -\frac{p}{p-q} -1 +o(1)} \quad \text{as } s\uparrow 1 \, ,\]
which can be turned into $p_n = n^{-1} L(|b_n|)  (\ell^\star(b_n))^{-1- \frac{p}{p-q} +o(1)}$ as done in Section~\ref{sec:p<q}. Again, as in Section~\ref{sec:p<q}, the term $\ell^\star(|b_n|)^{o(1)}$ can be replaced by $\tilde L (\ell^\star (|b_n|))$ for some slowly varying function $\tilde L$, under the assumption V1-V2.

\smallskip
\textit{(ii) The case $p<q$.}
Here also, the method of Section~\ref{sec:p>q} is easily adapted. Denote $h(s)$ as in \eqref{def:h}. Then Lemma~\ref{lem:P<0bis} gives that $\sum_{k=1}^n k^{-1} \bP(S_k <0)\sim - \frac{q}{q-p} \log \ell^{\star}(b_n)$ as $n\to+\infty$, so that  Corollary~1.7.3 in \cite{cf:BGT}  gives that $h(s) \sim  -\frac{q}{q-p} \log \ell^\star(b_{1/(1-s)})$ as $s\uparrow1$. Therefore, we get
\[ p(s) =\frac{e^{-h(s)}}{1-s} = \frac{1}{1-s} \Big( \ell^\star(b_{1/(1-s)}) \Big)^{q/(q-p) + o(1)} \, ,\] 
which can be turned into $p_n =  (\ell^\star(b_n))^{q/(q-p) +o(1)}$ as done in Section~\ref{sec:p>q}. Here again, $\ell^\star(|b_n|)^{o(1)}$ can be replaced by $\bar L (\ell^\star (|b_n|))$ for some slowly varying function $\bar L$, under the assumption V1-V2.

\smallskip
\textit{(iii) The case $p=q$.} In the case $\lim_{n\to+\infty} b_n/a_n =b\in\mathbb{R}$, the proof is identical to that of Section~\ref{sec:pqbalanced}: we get that
\[ \lim_{n\to+\infty} \frac{1}{n} \sum_{k=1}^n \bP(S_k>0) = \bP(Y> -b) =:\rho \in(0,1),  \]
with $\rho = \frac12 +\frac{1}{\pi} \arctan ( 2b/\pi )$. This is Spitzer's condition, and \cite[Thm.~8.9.12]{cf:BGT} implies that $T_-$ is in the domain of attraction of a positive stable random variable with index $\rho$, so that there exists a slowly varying function $\varphi(\cdot)$ such that
\begin{equation}
\label{eq:pqmu0}
\bP(T_- >n) \simn \gp(n) n^{-\rho}.
\end{equation}

The case $\lim_{n\to+\infty} b_n/a_n = +\infty$ (or $-\infty$) can be treated similarly to Section~\ref{sec:p=q}: we get the same conclusion as in  \eqref{eq:T-T+}.

\section{Renewal theorems: proof of Theorems \ref{thm:renewcentered}-\ref{thm:renewal}-\ref{thm:renewalfinite}}
\label{sec:renewals}

\subsection{The case $\ga=1$ with infinite mean}

\subsubsection{The case $\lim_{n\to+\infty} b_n/a_n =b \in \bbR$}
Recall that we have $p=q=1/2$ in that case.
Let us set $k_x$ an integer such that $a_{k_x}\sim x$. We fix $\gep>0$ such that $\sup_n |b_n|/a_n \le 1/\sqrt{\gep}$, and we write
\begin{align}
\label{eq:splitG}
G(x) := \Big(  \sum_{k \le \frac 1\gep k_x}  + \sum_{k> \tfrac1\gep k_x}  \Big) \bP(S_k =x).
\end{align}
The first term is estimated thanks to Theorem \ref{thm:localLD}, which gives that there is a constant $C$ such that for any $x\ge 1$ and any $k$, $\bP(S_k =x) \le C (a_k)^{-1}  k L(x)x^{-1}$, so that 
\begin{equation}
\label{eq:partG1}
\sum_{k \le  \frac1 \gep k_x} \bP(S_k=x)  \le C' \Big(\frac{k_x}{\gep} \Big)^2  (a_{k_x/\gep})^{-1} L(x) x^{-1} \le  \frac{C''}{ \gep}  L(x)^{-1} .
\end{equation}
where we used that $a_{k_x/\gep} \sim \gep^{-1} a_{k_x}$ and that $k_x \sim a_{k_x} L(a_{k_x})^{-1} \sim x L(x)^{-1}$.

The second term is in fact the main one.
Thanks to the local limit theorem \eqref{eq:LLT}, for any $\eta>0$ there is some $k_0$ such that for any $k\ge k_0$,
\[ g\big( (x- b_k)/a_k \big) - \eta \le a_k \bP(S_k =x) \le   g\big( (x-b_k)/a_k \big) +\eta \, .\]
Then for any $\gep>0$ fixed, $|(x- b_k)/a_k -b| \le |x/a_k| + |b_{k}/a_k -b|\le 5\gep$, provided that  $k > \gep^{-1} k_x$ and $x$ is large enough. 
Indeed, since $a_k$ is regularly varying with index $-1$, for $k> \gep^{-1}k_x$ we have $a_{k} \ge \frac12 \gep^{-1} a_{k_x} \ge x/(4\gep)$. Since $|b_k/a_k -b| \le \gep$ for $k$ large enough, we obtain the above claim. Therefore, by continuity of $g$ we get that, provided  $\gep$ is small and $x$ is large enough
\[ g( -b ) - 2\eta \le a_k \bP(S_k =x) \le   g( -b ) +2\eta \qquad \text{ for all } k\ge  \gep^{-1} k_x .\]
Since we are in the symmetric case, with $p=q=1/2$, and by our definition \eqref{def:an} of $a_n$, $g(\cdot)$ is the density of a symmetric Cauchy($1/2$) distribution, so that $g(-b) =\frac{2}{\pi(1+(2b)^2)}$.

Hence, for any $\eta'>0$ and provided that $\gep$ is small enough and $x$ large enough (so that $\gep^{-1} k_x$ is large), the second sum in \eqref{eq:splitG} is
\[\frac{2(1-\eta')}{\pi(1+(2b)^2) } \sum_{k>\frac1\gep k_x} \frac{1}{a_k} \le\sum_{k> \tfrac1\gep k_x}  \bP(S_k =x)  \le  \frac{2(1+\eta')}{\pi(1+(2b)^2) } \sum_{k>\frac1\gep k_x} \frac{1}{a_k} .\]
We then estimate the last sum thanks to a comparison with the following (convergent) integral
\[ \int_{ k_x/\gep}^{+\infty} \frac{dt}{a_t}  \sim \int_{a_{ k_x/\gep}}^{+\infty} \frac{du}{uL(u)}  \sim \int_{x/\gep}^{+\infty} \frac{du}{u L(u)}\quad \text{ as } x\to+\infty. \]
We used a change of variable $u=a_t$ so that $t\sim u/L(u)$ (see \eqref{def:an}) and $dt= L(u)^{-1}du$, and then used that $a_{ k_x/\gep } \sim \gep^{-1} a_{k_x} \sim \gep^{-1} x$.
Since $v\mapsto \int_{v}^{+\infty} \frac{du}{uL(u)}$ is a  slowly varying function (vanishing as $v\to+\infty$), we get that for  $\gep>0$ small enough  and $x$ large enough  (how large depends on $\gep$)
\begin{equation}
\label{eq:partG2}
\frac{2 (1-2\eta')}{\pi(1+(2b)^2)} \sum_{n>x} \frac{1}{nL(n)} \le \sum_{k> \tfrac1\gep k_x}  \bP(S_k =x) \le  \frac{2 (1+2\eta')}{\pi(1+(2b)^2)} \sum_{n>x} \frac{1}{nL(n)} .
\end{equation}

In conclusion, combining \eqref{eq:partG1} and \eqref{eq:partG2}, and since $ L(x)^{-1}=o\big(\sum_{n>x} \frac{1}{nL(n)}  \big)$ and $\eta'$ is arbitrary, we get that
\[G(x) \sim \frac{2}{\pi(1+(2b)^2)}   \sum_{n>x} \frac{1}{n L(n)} \quad \text{as } x\to+\infty\, .\]

\subsubsection{The case $p>q$}
Let  $k_x$ to be a solution of $b_{k_x} = k_x\mu(a_{k_x}) =x$ (in the following, we assume for simplicity of notation that $k_x$ is an integer).
Then we identify the range of $k$'s for which we may apply the local limit theorem \eqref{eq:LLT} to $\bP(S_k =x)$: they are the $k$'s such that $x- b_k$ is of order $a_k$,  and we find that they are in the range $k = k_x + \Theta \big( a_{k_x} /\mu(a_{k_x}) \big)$.
Let us mention the results of \cite{cf:AA,cf:HR} where this heuristic is confirmed: if $N_x$ the number of renewals before reaching $x$, it is shown that $(a_{k_x}/\mu(a_{k_x}))^{-1}(N_x-k_x) $ converges in distribution. 

Let us stress right away that $\mu(a_{k_x}) \sim \mu(b_{k_x})= \mu(x)$. Indeed, since $p>q$ we have that $\mu(x) \sim (p-q) \ell(x)$, and Lemma~\ref{lem:ell}  gives that $\ell(a_n) \sim \ell(b_n)$.

We fix $\gep>0$ and decompose $G(x)$ into five sums
\begin{align}
\label{eq:splitG5}
G(x) &= \Big( \sum_{k < \tfrac12  k_x} + \sum_{k =\tfrac12 k_x}^{k_x-  \frac{a_{k_x}}{\gep \mu(x)}}  + \sum_{k=k_x- \frac{a_{k_x}}{\gep \mu(x)} }^{k_x + \frac{a_{k_x}}{\gep \mu(x)} } + \sum_{k_x +\frac{a_{k_x}}{\gep \mu(x)} }^{2k_x} + \sum_{k>2k_x} \Big) \bP(S_k=x) \notag\\
 &=:  \quad \ {\rm I}  \ \quad + \quad {\rm  II}\quad\  +\ \quad  {\rm III}  \quad\ \ +\quad {\rm   IV}\quad  +\quad {\rm  V} \, .
\end{align}
The main contributions are the sums III and V, so we start by estimating those  terms.
\smallskip

{\bf Term III.}
By the local limit theorem \eqref{eq:LLT}, we get that as $x\to +\infty$ (so $k_x\to +\infty$)
\begin{align*}
{\rm III } &  = (1+o(1))   \sum_{k =  k_x - \frac{a_{k_x}}{\gep \mu(x)} }^{k_x+ \frac{a_{k_x}}{\gep \mu(x)} }  \frac{1}{a_{k}}   g\Big(  \frac{x- k\mu(a_k)}{a_k}\big) \, .
\end{align*}
Then, we use the fact that $a_{k_x}$ is negligible compared to $b_{k_x} =k_x \mu(a_{k_x}) \sim k_x \mu(x) $: we get that uniformly for the $k$'s in the range considered, we have $k=(1+o(1))k_x$ so that $a_k  =(1+o(1)) a_{k_x}$. Setting $j=k-k_x$, we also have that for the range of $k$ considered (using also $\mu(a_{k_x})\sim \mu(x)$), since $x=k_x \mu(a_{k_x})$
\begin{equation}
\label{eq:x-bk}
\frac{x-k \mu(a_k)}{a_{k}} = (1+o(1)) \frac{j \mu(x)}{a_{k_x}} + \frac{k_{x}  \big( \mu(a_{k_x}) - \mu(a_k) \big) }{(1+o(1)) a_{k_x}}  =(1+o(1)) \frac{j \mu(x)}{a_{k_x}} +o(1). 
\end{equation}
For the second identity, we used Claim \ref{claim:mu} to get that $|\mu(a_{k_x}) - \mu(a_k)| = o(L(a_{k_x})) = o(a_{k_x}/k_x)$.
In the end, and since $g$ is continuous, we get that
\begin{align}
{\rm  III } & = (1+o(1))  \sum_{j= - \frac{a_{k_x}}{\gep \mu(x)}}^{ \frac{a_{k_x}}{\gep \mu(x)}} \frac{1}{a_{k_x}} g\Big( j \times  \frac{\mu(x)}{a_{k_x}}\Big) = \frac{1+o(1)}{\mu(x)} \int_{-1/\gep}^{1/\gep} g(u) du  \, ,
\end{align}
where we used a Riemann sum approximation in the last identity.
Since $\int_{-\infty}^{+\infty} g(u) du =1$, we then get that for any $\eta>0$ we can choose $\gep>0$ such that for all sufficiently large $x$ (how large depend on $\gep$)
\begin{equation}
\label{termIII}
\frac{1-\eta}{\mu(x)}  \le {\rm  III } \le \frac{1+\eta}{\mu(x)} \, .
\end{equation}

\smallskip
{\bf Term V.}
For the last term in \eqref{eq:splitG5}, we use \eqref{hyp:localtail2}: Theorem \ref{thm:localDoney} gives that for $k\ge 2 k_x$
\[\bP(S_k=x) = \bP(S_k -b_k = x - b_k) \le C k L(b_k) b_k^{-(1+\ga)} , \]
where we used that $b_k \ge b_{2k_x} \ge \frac32 b_{k_x} = \frac32 x $ provided that $x$ is large enough, so that $|x-b_k| \ge  \frac12 b_k \gg a_k$.
Then we get that (we have $\ga=1$)
\begin{equation}
\label{termVa}
{\rm V} \le C  \int_{k_x}^{+\infty} \frac{u  L(b_u)}{b_u^2} {\rm d} u \le C' \int_{x}^{+\infty}  \frac{L(t)}{t \ell(t)^2} {\rm d} t  = \frac{C'}{\ell(x)} \, ,
\end{equation}
where  we used a change of variable $t=b_u \sim (p-q) u \ell(b_u)$ (using also $dt \sim  (p-q)\ell(b_u) du$ and $b_{k_x}=x$), and then Lemma~\ref{lem:sumell}-(i).

If one has additionally that $\bP(X_1=-x) \sim q L(x) x^{-2}$, then we write for any  $\gep \in (0,1/2)$
\[{\rm V} =\sum_{k=2k_x}^{ \frac1\gep k_x} \bP(S_k=x) + \sum_{k > \frac1\gep k_x } \bP(S_k=x)  \, .\]
As above, the first term is comparable to $\int_{2k_x}^{ k_x/\gep}  \frac{u  L(b_u)}{b_u^2} {\rm d} u\le C \int_{x}^{ x/\gep}  \frac{  L(t)}{t \ell(t)^2} {\rm d} t$, which is $o(1/\ell(x))$ because of Lemma \ref{lem:sumell} and since $\ell(x)$ is slowly varying (so  $\ell( x/\gep) \sim \ell(x)$).
For the second term, we use Theorem \ref{thm:localDoney} which gives that 
for any $\eta>0$, and provided that $\gep$ is fixed small enough and that $x$ is large enough, we have for all $k\ge \gep^{-1} k_x$ (so $x= b_{k_x} \le \frac{1}{2 \gep} b_k$)
\[ \bP(S_k =x) = \bP(S_{k}- b_k = x- b_k) 
\begin{cases}
\le (q+\eta) k L(b_k) b_k^{-2}\, , \\
\ge (q-\eta) k L(b_k) b_k^{-2}\, .
\end{cases}
\]
Moreover,  by a change of variable $t = b_u$ ($dt \sim (p-q)\ell(b_u) du$) we get
\begin{equation}
\label{restsum}
\int_{ k_x/\gep}^{+\infty} \frac{u L(b_u) }{b_u^2} {\rm d} u \sim  \int_{ x/\gep}^{+\infty}  \frac{1}{(p-q)^2} \frac{L(t) }{t \ell(t)^2}  {\rm d} t  =  \frac{1}{(p-q)^2} \frac{1}{\ell( x/\gep)} \sim \frac{1}{(p-q) \mu(x)} ,
\end{equation}
where we used that $b_{ k_x/\gep} \sim  b_{k_x}/\gep = x/\gep$ and that $\ell(\cdot)$ is slowly varying, with $\mu(x)\sim (p-q) \ell(x)$.
In the end, and since $\eta$ is arbitrary, we get that as $x\to+\infty$,
\begin{equation}
\label{termV}
{\rm V} = (1+o(1)) \frac{q}{p-q} \frac{1}{\mu(x)} \, .
\end{equation}

To conclude the proof of the statement, we need to show that the terms I, II and IV are negligible compared to $1/\mu(x)$ (or $1/\ell(x)$).

\smallskip
{\bf Term I.} Thanks to \eqref{hyp:localtail1} and Theorem \ref{thm:localDoney}, we obtain that there is a constant $C>0$ such that for any $k\le  k_x/2$ (so that $x \ge \frac23 b_{k} \gg a_k$ provided that $x$  is large engouh)  we have $\bP(S_k =x) \le C  k L(x) x^{-2}$.
Then the first term in \eqref{eq:splitG5} is bounded by a constant times
\begin{equation}
\label{termI}
 \sum_{k=1}^{ k_x /2 } k L(x) x^{-2} \le  k_x^2 L(x) x^{-2} \le \frac{c}{\ell(x)} \frac{L(x)}{\ell(x)}.
\end{equation}
For the last inequality, we used that $k_x\sim \frac{1}{p-q} x\ell(x)^{-1}$ as $x\to+\infty$. (Indeed $x=b_{k_x} \sim (p-q) k_x \ell(b_{k_x})$.)
Now, because $L(x)/\ell(x) \to 0$, we get that ${\rm I} = o(1/\ell(x))$.

\smallskip
{\bf Term II.}
We set $j=k_x-k$. Then, the range of $k$ considered corresponds to $j \in \big[ \frac{a_{k_x}}{\gep \mu(x)}, \frac12 k_x \big]$, and for that range we have similarly to \eqref{eq:x-bk}
\begin{equation}
\label{eq:x-bk2}
\frac{x-b_k}{a_k} =  \frac{k_x \mu(a_{k_x})  - (k_x -j)\mu(a_{k_x -j}) }{ a_{k}} \ge \frac{ k_x  }{ a_{k}} \big| \mu(a_{k_x}) - \mu(a_{k_x -j})  \big| + j \frac{ \mu(x)}{2 a_{k}} \, .
\end{equation}
We used that $\mu(a_{k_x -j} ) \sim \mu(a_{k_x}) \sim \mu(x)$ (since $j\le k_x/2$ and $\mu(\cdot)$ is slowly varying).
Then, we may use Claim \ref{claim:mu} to get that 
\[| \mu(a_{k_x}) - \mu(a_{k_x -j})| \le c L(a_{k_x}) \log \big(a_{k_x} / a_{k_x-j} \big) \le c' L(a_{k_x})  j/k_x  \, ,\]
where in the second inequality we used that there is a constant $c>0$ such that uniformly for $j\in[k/2,k]$, $a_k /a_{k-j} \le 1 + c j /k$ (using  that $a_k$ is regularly varying with index $1$). Plugging this in \eqref{eq:x-bk2}, and since $L(a_{k_x}) = o(\mu(a_{k_x})) = o(\mu(x))$, we get that
\begin{equation}
\label{eq:boundx-bk}
x-b_k \ge \frac14 j \mu(x) , \qquad \text{ for } j := k_x-k\in \Big[ \frac{a_{k_x}}{\gep \mu(x)} , \frac12 k_x \Big] \, .
\end{equation}
Then, since  we have here $j \mu(x) \ge \gep^{-1} a_{k_x}$ we get that $x-b_k \ge \frac14 j\mu(x) \ge a_k$ ($k\ge k_x/2$), so that we may apply Theorem \ref{thm:localDoney}.
We therefore obtain, for $j:=k_x-k \in \big[ \frac{a_{k_x}}{\gep \mu(x)}, \frac12 k_x \big]$
\begin{align}
\bP(S_k=x) = \bP(S_k -b_k = x-b_k)& \le C k L(x-b_k) (x-b_k)^{-2}\notag \\
&  \le C k_x   L( j \mu(x) ) (j \mu(x))^{-2}\, .
\label{applyDoney}
\end{align}
Overall, we get that 
\begin{align*}
{\rm  II} & = \sum_{k=\frac12 k_x}^{ k_x - \frac{a_{k_x}}{\gep \mu(x)}}  \bP(S_k= x) \le C k_x \int_{\frac{a_{k_x}}{\gep \mu(x)}}^{ \frac12  k_x}   \frac{ L(j \mu(x) ) }{  (j \mu(x) )^{2} }  {\rm d}j\\
 & \le C  \frac{ k_x}{\mu(x)} \int_{  \gep^{-1} a_{k_x} }^{+\infty} \frac{L(t )}{ t^2} {\rm d}t \le C \frac{k_x}{\mu(x)} \frac{L( \gep^{-1}a_{k_x})}{ \gep^{-1} a_{k_x}} \, ,
\end{align*}
where in the second inequality we made a change of variable $t = j\mu(x)$.
Now, since $L(\cdot)$ is slowly varying, and because of the definition \eqref{def:an} of $a_n$, we get that $k_x L(\gep^{-1} a_{k_x}) /a_{k_x} \to 1$ as $x\to+\infty$.
Hence, there exists a constant $C>0$ (independent of $\gep$) such that for $x$ sufficiently large (how large depends on $\gep$) we have
\begin{equation}
\label{termII}
{\rm II } \le \frac{C \gep}{\mu(x)} \, .
\end{equation}

\smallskip
{\bf Term IV.} It is treated similarly to the term II.
Setting $j= k -k_x$, one gets exactly as in \eqref{eq:boundx-bk} that provided that $x$ is large enough,
\begin{equation}
\label{eq:boundx-bk2}
x-b_k \le  - \frac14 j \mu(x)  \, , \qquad \text{ for } j := k-k_x\in \Big[  \frac{a_{k_x}}{\gep\mu(x)} , 2k_x \Big] \, .
\end{equation}
Then, we may apply Theorem~\ref{thm:localDoney}  (we have that $|x-b_k| \ge a_k$ for the range considered) to get analogously to \eqref{applyDoney} that there is some $C$ such that for any $j:=k-k_x \in \big[  \frac{a_{k_x}}{\gep\mu(x)} , 2k_x \big]$
\[\bP(S_k = x) = \bP(S_k -b_k = x-b_k)  \le C k_x L(j \mu(x)) \big( j\mu(x) \big)^2\,. \]
Therefore, we can bound the term IV
\begin{align*}
{\rm IV} = \sum_{k=k_x + \frac{a_{k_x}}{\gep\mu(x)} }^{2k_x} \bP(S_k=x) \le C k_x \sum_{j=\frac{a_{k_x}}{\gep\mu(x)}}^{2 k_x} \frac{L(j \mu(x))}{ \big( j \mu(x) \big)^2} \, .
\end{align*}
Then as for the term II, we get that there is a constant $C$ (independent of $\gep$) such that for $x$ sufficiently large 
\begin{equation}
\label{termIV}
{\rm IV} \le \frac{C \gep}{\mu(x)} \, .
\end{equation}
 
\noindent
{\bf Conclusion.}
Assuming \eqref{hyp:localtail1}-\eqref{hyp:localtail2}, we get from  \eqref{termIII}-\eqref{termVa} and \eqref{termI}-\eqref{termII}-\eqref{termIV} that  there is a constant $C$ such that 
$G(x) \le C/\mu(x)$
holds for $x$ sufficiently large.

If we additionally assume that $\bP(X_1 = -x) \sim q L(x) x^{-1}$ as $x\to+\infty$, we can use \eqref{termV} instead of \eqref{termVa}.
According to \eqref{termIII}-\eqref{termII}-\eqref{termIV} and to \eqref{termV}, we find that for every $\eta>0$, we can choose $\gep>0$ small enough, such that for $x$ sufficiently large we get 
\begin{equation}
\frac{1- \eta}{\mu(x)} + (1+o(1))\frac{q}{p-q} \frac{1}{\mu(x)}  \le G(x)  \le \frac{1+ 3\eta}{\mu(x)} +  (1+o(1))\frac{q}{p-q}  \frac{1}{\mu(x)} .
\end{equation}
Since $\eta>0$ is arbitrary, we get \eqref{eq:renewalthm}.

\subsubsection{The case $p<q$}

Here, we have that $b_k =k \mu(a_k) \to -\infty$ as a regularly varying function, since $\mu(k)\sim (p-q) \ell(k)$.
We define  $k_x$ to verify $b_{k_x} =-x$.
Let us fix $\gep>0$, and split the sum in $G(x)$ into two parts this time.
\begin{equation}
\label{spliG3}
G(x) = \Big( \sum_{k=1}^{ \frac 1\gep k_x} + \sum_{k> \frac 1\gep k_x} \Big) \bP(S_k=x) =: {\rm I} + {\rm II}.
\end{equation}

\noindent
{\bf Term I.} 
Since $x-b_k \ge -b_k= |b_k|$ with $|b_k|\ge a_k$, we may  use \eqref{hyp:localtail1} to get thanks to Theorem~\ref{thm:localDoney} that
\begin{equation}
\label{eq:boundDoney}
\bP(S_k= x) = \bP(S_k - b_k = x-b_k) \le C  k \frac{L(x -b_k)}{(x-b_k)^2}\, .
\end{equation}
Then, since $x-b_k\ge x$ (except possibly for finitely many $k$'s for which $b_k$ is positive), we get that
$\bP(S_k=x) \le C k L(x) x^{-2}$ for all $k\le \gep^{-1} k_x$ (and $k$ larger than a constant).
Hence, we get that the term I in \eqref{spliG3} is bounded by
\begin{equation}
{\rm I} \le C   \big(k_x/\gep\big)^2 L(x)x^{-2}  \leq C'  \gep^{-2}\frac{L(x)}{ \ell(x)^2}\, ,
\end{equation}
where we used for the second inequality that $k_x \sim \frac{1}{q-p} x\ell(x)^{-1}$ (as for~\eqref{termI}).
Then, since $L(x)/\ell(x) \to 0$ as $x\to +\infty$, we get that
${\rm I} = o(1/\ell(x))$.

\smallskip
\noindent
{\bf Term II.}
Again, \eqref{eq:boundDoney} is valid. Here, we use that $x-b_k \ge |b_k|$ to get that $\bP(S_k =x) \le C k L(|b_k|) |b_k|^2$.
Then, we obtain that
\begin{equation}
{\rm II} \le C \sum_{k\ge \frac1\gep k_x} k \frac{ L(|b_k|) }{|b_k|^{2}} \le \frac{C}{\ell(x )} 
\end{equation}
where we used the same calculation as in \eqref{restsum}, with $|b_k| \sim (q-p) k \ell(|b_k|)$. 

Hence we proved that also in that case there is some $C>0$ such that $G(x) \le C/|\mu(x)|$ (since $|\mu(x)| \sim -(q-p)\ell(x) $). 

Let us now obtain the right asymptotic equivalence, assuming \eqref{hyp:doney}.  We may use Theorem~\ref{thm:localDoney} to obtain that for any $\eta>0$, we can choose $\gep>0$ so that if $x$ is large enough we get for any $k\ge \gep^{-1} k_x$ we have $|b_k| \le x-b_{k} \le (1+2\gep) |b_k|$ and then
\[
\bP(S_k =x) = \bP(S_k =b_k = x-b_k) 
\begin{cases}
\le (p+\eta) k L(|b_k|) |b_k|^{-2}  \, , \\
\ge (p-\eta) k L(|b_k|) |b_k|^{-2} \, .
\end{cases}
\]
The same calculation as in \eqref{restsum} then gives that
\begin{equation}
\sum_{k> \frac1\gep k_x} \frac{k L(|b_k|)}{|b_k|^2} \sim \frac{1}{(q-p) |\mu(x)|},
\end{equation}
using that $|b_k| \sim (q-p) k \ell(|b_k|)$ and $|\mu(x)| \sim (q-p) \ell(x)$.
Since $\eta$ is arbitrary, we get that as $x\to+\infty$
\begin{equation}
{\rm V}= \sum_{k>\gep^{-1} k_x} \bP(S_k =x) \stackrel{x\to+\infty}{\sim} \frac{p}{q-p} \frac{1}{|\mu(x)|} \, .
\end{equation}

\subsection{The finite mean case: proof of Theorem~\ref{thm:renewalfinite}}
\label{sec:finitemeanrenewal}

We only have to treat the case $\mu<0$. Recall that $S^* = \sup_{i \ge 0} S_i$, and for $x>0$ let $\tau_x= \inf\{n\geq 1, S_n\geq x\}$.
Then, using Markov's property, we have that
\begin{align*}
G(x) = \sum_{k=1}^{+\infty} \sum_{y=x}^{+\infty} \bP(\tau_x =k, S_{\tau_x} =y) G(y,x)
\end{align*}
with $G(y,x)=\sum_{n=0}^{+\infty} \bP(y+S_n =x) =G(x-y)$. The renewal theorem (Theorem~\ref{thm:renewalfinite}-(i)) gives that $G(x-y) \to 1/|\mu|$ as $x-y \to -\infty$, and also $G(y-x)\le C$ for all $y\ge x$.
Then, we take some $\gep_x>0$ with $\gep_x\to 0$ and $x\gep_x \to+\infty$, and we split the sum according to whether $y\in [x,(1+\gep_x)x]$ or $y>(1+\gep_x)x$. The first term is
\begin{align*}
\sum_{k=1}^{+\infty} \sum_{y\in [x, (1+\gep_x) x]} \bP(\tau_x =k, S_{\tau_x} =y) G(x-y) & \le C\sum_{k=1}^{+\infty} \sum_{y\in [x, (1+\gep_x) x]} \bP(\tau_x =k, S_{\tau_x} =y) \\
& = C \bP\big( S^* \in [x,(1+\gep)x]\big)\,.
\end{align*}
The second term deals with $y>(1+\gep_x)x$, so $x-y\to -\infty$: it is
\begin{align*}
\sum_{k=1}^{+\infty} \sum_{y> (1+\gep_x) x} \bP(\tau_x =k, S_{\tau_x} =y) G(x-y) & = \sum_{k=1}^{+\infty}  \sum_{y> (1+\gep_x) x} \bP(\tau_x =k, S_{\tau_x} =y) \frac{1+o(1)}{|\mu|}\\
& = \frac{1+o(1)}{|\mu|}\bP\big( S^* > (1+\gep_x) x\big)\,.
\end{align*}
As a consequence, we get that
\begin{equation}
G(x) = (1+o(1)) \frac{1}{|\mu|}\bP\big( S^* >  x\big) + O\Big( \bP(S^*>x) - \bP\big( S^* >(1+\gep_x)x\big)\Big).
\end{equation}
Since $\bP(S^* >x)$ is subexponential, the second term is negligible compared to the first one.\qed

\bigskip
\noindent
{\bf Acknowledgements:} I am most grateful to Vitali Wachtel for his comments and his suggestions for the improvement of Theorem \ref{thm:ladder}, and also to I. Kortchemski and L. Richier for attracting my attention to some subtleties of Theorem \ref{thm:laddermu0} (and to their article \cite{cf:KR}).
I~thank the referees for their remarks and suggestions, in particular for pointing out references (and an elementary proof) for Theorem~\ref{thm:renewalfinite}.
I am grateful to Thomas Duquesne and Francesco Caravenna for the many discussions we had on this (and related) topic(s), and I also thank Cyril Marzouk for telling me about the reference~\cite{cf:BCM}.

\end{document}